\documentclass[12pt]{amsart}
\usepackage{color}

\usepackage{amssymb}
\ifx\pdfoutput\undefined
\usepackage{graphicx}  
\DeclareGraphicsExtensions{.pstex,.eps}

\else
\usepackage{amsfonts}
\usepackage[pdftex]{graphicx}  \DeclareGraphicsExtensions{.pdf,.mps}

\fi

\newcommand{\RR}{{\mathbb R}}

\newcommand{\spt}{\operatorname{supp}}

\renewcommand{\Re}{\mathop{\rm Re}\nolimits}
\renewcommand{\Im}{\mathop{\rm Im}\nolimits}

\newtheorem{thm}{Theorem}
\newtheorem{prop}{Proposition}[section]
\newtheorem{cor}[prop]{Corollary}
\newtheorem{lem}[prop]{Lemma}
\newtheorem{lemma}[prop]{Lemma}
\newtheorem{claim}[prop]{Claim}

\newtheorem{rem}{Remark}[section]
\newtheorem{defn}[prop]{Definition}

\numberwithin{equation}{section}
\numberwithin{prop}{section}

\def\squarebox#1{\hbox to #1{\hfill\vbox to #1{\vfill}}}

\def\R{\mathbb R}
\def\reals{\mathbb R}

\def\reals{{\mathbb R}}

\def\Re{\,\mathrm{Re}\,}
\def\Im{\,\mathrm{Im}\,}

\def\be{\begin{eqnarray*}}
\def\ee{\end{eqnarray*}}
\def\ben{\begin{eqnarray}}
\def\een{\end{eqnarray}}
\usepackage{amsxtra}

\title
[Dispersive estimates for matrix Hamiltonians]
{Dispersive estimates using scattering theory for matrix Hamiltonian equations}

\author[J.L. Marzuola]
{Jeremy L. Marzuola}

\address{Applied Mathematics Department, Columbia University \\
200 S. W. Mudd, 500 W. 120th St., New York City, NY 10027, USA}

\begin{document}    
   
\begin{abstract}
We develop the techniques of \cite{KS1} and \cite{ES1} in order to derive dispersive estimates for a matrix Hamiltonian equation defined by linearizing about a minimal mass soliton solution of a saturated, focussing nonlinear Schr\"odinger equation
\begin{eqnarray*}
\left\{\begin{array}{c}
i u_t + \Delta u + \beta (|u|^2) u = 0 \\
u(0,x) = u_0 (x),
\end{array} \right.
\end{eqnarray*}
in $\reals^3$.  These results have been seen before, though we present a new approach using scattering theory techniques.  In further works, we will numerically and analytically study the existence of a minimal mass soliton, as well as the spectral assumptions made in the analysis presented here.  
\end{abstract}

\maketitle 

\section{Introduction}

In this result, we develop the dipsersive estimates used to prove stability of solitons for a focussing, saturated nonlinear Schr\"odinger equation (NLS) in $\R \times \R^d$:
\begin{eqnarray*}
i u_t + \Delta u + \beta (|u|^2) u & = & 0 \\
u(0,x) & = & u_0 (x),
\end{eqnarray*}
where $\beta: \reals \to \reals$, $\beta(s) \geq 0$ for all $s \in \reals$, $\beta$ has a specific structure outlined in the following definitions:

\begin{defn}
\label{def:non1}
Saturated nonlinearities of type $1$ are of the form
\begin{eqnarray}
\label{sat:eqn1}
\beta (s) = s^{\frac{q}{2}} \frac{s^{\frac{p-q}{2}}}{1 + s^{\frac{p-q}{2}}}, 
\end{eqnarray}
where $p > 2+\frac{4}{d}$ and $\frac{4}{d}  > q > 0$ for $d \geq 3$ and $\infty > p > 2+\frac{4}{d} > \frac{4}{d} > q > 0$ for $d < 3$.
\end{defn}    

\begin{defn}
\label{def:non2}
Saturated nonlinearities of type $2$ are of the form
\begin{eqnarray}
\label{sat:eqn2}
\beta(s) = \frac{s}{(1 + s)^{\frac{2-q}{2}}},
\end{eqnarray}
where $\frac{4}{d} > q >0$, $d > 2$. 
\end{defn}

\begin{rem}
In both cases, for $|u|$ large, the behavior is $L^2$ subcritical and for $|u|$ small, the behavior is $L^2$ supercritical.  For Definition \ref{def:non1}, $p$ is chosen much larger than the $L^2$ critical exponent, $\frac{4}{d}$ in order to allow sufficient regularity when linearizing the equation about the soliton.  
\end{rem}

In the sequel, we assume that $u_0 \in H^1$ and $|x|u_0 \in L^2$, or in other words, $u_0$ has finite variance.  For initial data with this regularity, from the spatial and phase invariance of NLS, we have many the following conserved quantities:

Conservation of Mass (or Charge):
\begin{eqnarray*}
Q (u) = \frac{1}{2} \int_{\R^n} |u|^2 dx = \frac{1}{2} \int_{\R^d} |u_0|^2 dx,
\end{eqnarray*}

and

Conservation of Energy:
\begin{eqnarray*}
E(u) = \int_{\R^d} | \nabla u |^2 dx - \int_{\R^d} G(|u|^2) dx = \int_{\R^d} | \nabla u_0 |^2 dx - \int_{\R^d} G(|u_0|^2) dx,
\end{eqnarray*}
where 
\begin{eqnarray*}
G(t) = \int_0^t \beta(s) ds.
\end{eqnarray*}

We also have the pseudoconformal conservation law:
\begin{eqnarray}
\| (x + 2 i t \nabla ) u \|^2_{L^2} - 4 t^2 \int_{\R^d} G(|u|^2) dx = \| x \phi \|^2_{L^2} - \int_0^t \theta (s) ds,
\end{eqnarray}
where 
\begin{eqnarray*}
\theta (s) = \int_{\R^d} (4 (d+2) G(|u|^2) - 4 d \beta(|u|^2) |u|^2) dx.
\end{eqnarray*}
Note that $(x + 2 i t \nabla )$ is the Hamilton flow of the linear Schr\"odinger equation, so the above identity shows how the solution to the nonlinear equation is effected by the linear flow.

Detailed proofs of these conservation laws can be arrived at easily using energy estimates or Noether's Theorem, which relates conservation laws to symmetries of an equation.  Global well-posedness in $L^2$ of (NLS) with $\beta$ of type $1$ or $2$ for finite variance initial data follows from standard theory for $L^2$ subcritical monomial nonlinearities.  Proofs of the above results can be found in numerous excellent references for (NLS), including \cite{Caz} and \cite{SS}.

{\sc Acknowledgments.} This paper is a result of a thesis done under the direction of Daniel Tataru at the University of California, Berkeley.  It is more than fair to say this work would not exist without his assistance.  Also, the author must specifically thank Maciej Zworski for pointing him towards the results of H\"ormander on perturbations of elliptic operators, as well as many other helpful conversations about this result.  The work was supported by Graduate Fellowships from the University of California, Berkeley and NSF grants DMS0354539 and DMS0301122.  In addition, the author spent a semester as an Associate Member of MSRI during the development of these results.  Currently, the author is supported by an NSF Postdoctoral Fellowship.

\section{Soliton Solutions}

A soliton solution is of the form 
\begin{eqnarray*}
u(t,x) = e^{i \lambda t} R_\lambda(x)
\end{eqnarray*} where 
$\lambda > 0$ and $R_\lambda (x)$ is a positive, radially symmetric, exponentially decaying solution of the equation:
\begin{eqnarray}
\label{eqn:sol}
\Delta R_\lambda - \lambda R_\lambda + \beta (R_\lambda ) R_\lambda = 0.
\end{eqnarray}
With nonlinearities of type $1$ or $2$, soliton solutions exist and are known to be unique.  Existence of solitary waves for nonlinearities of the type presented in Definitions \ref{def:non1} and \ref{def:non2} is proved by in \cite{BeLi} by minimizing the functional 
$$T(u) = \int | \nabla u |^2 dx$$ 
with respect to the functional 
$$V(u) = \int [ G(|u|^2) - \frac{\lambda}{2} |u|^2 ] dx.$$  
Then, using a minimizing sequence and Schwarz symmetrization, one sees the existence of the nonnegative, spherically symmetric, decreasing soliton solution.  For uniqueness, see \cite{Mc}, where a shooting method is implemented to show that the desired soliton behavior only occurs for one particular initial value. 

An important fact is that $Q_{\lambda} = Q(R_{\lambda})$ and $E_{\lambda} = E(R_{\lambda})$ are differentiable with respect to $\lambda$.  This fact can be determined from the early works of Shatah, namely \cite{Sh1}, \cite{Sh2}.  By differentiating Equation \eqref{eqn:sol}, $Q$ and $E$ with respect to $\lambda$, we have
\begin{eqnarray*}
\partial_{\lambda} E_{\lambda} = - \lambda \partial_{\lambda} Q_{\lambda}.
\end{eqnarray*}

Numerics show that if we plot $Q_{\lambda}$ with respect to $\lambda$, we get a curve that goes to $\infty$ as $\lambda \to 0, \infty$ and has a global minimum at some $\lambda = \lambda_0 > 0$, see Figure \ref{f:solcurves}.  We will explore this in detail in a subsequent numerical work \cite{Mnum}.  Variational techniques developed in \cite{GSS} and \cite{ShSt} tell us that when $\delta ( \lambda ) = E_{\lambda} + \lambda Q_{\lambda}$ is convex, or $\delta '' (\lambda) > 0$, we are guaranteed stability under small perturbations, while for $\delta '' (\lambda) < 0$ we are guaranteed that the soliton is unstable under small perturbations.  We will explore the nature of this stability in another subsequent work, \cite{Mnl}, where we study the full nonlinear problem.  For a brief reference on this subject, see \cite{SS}, Chapter 4.  For nonlinear instability at a minimum, see \cite{CP1}.  For notational purposes, we refer to a minimal mass soliton as $R_{min}$.

\begin{figure}
\scalebox{0.38}{\includegraphics{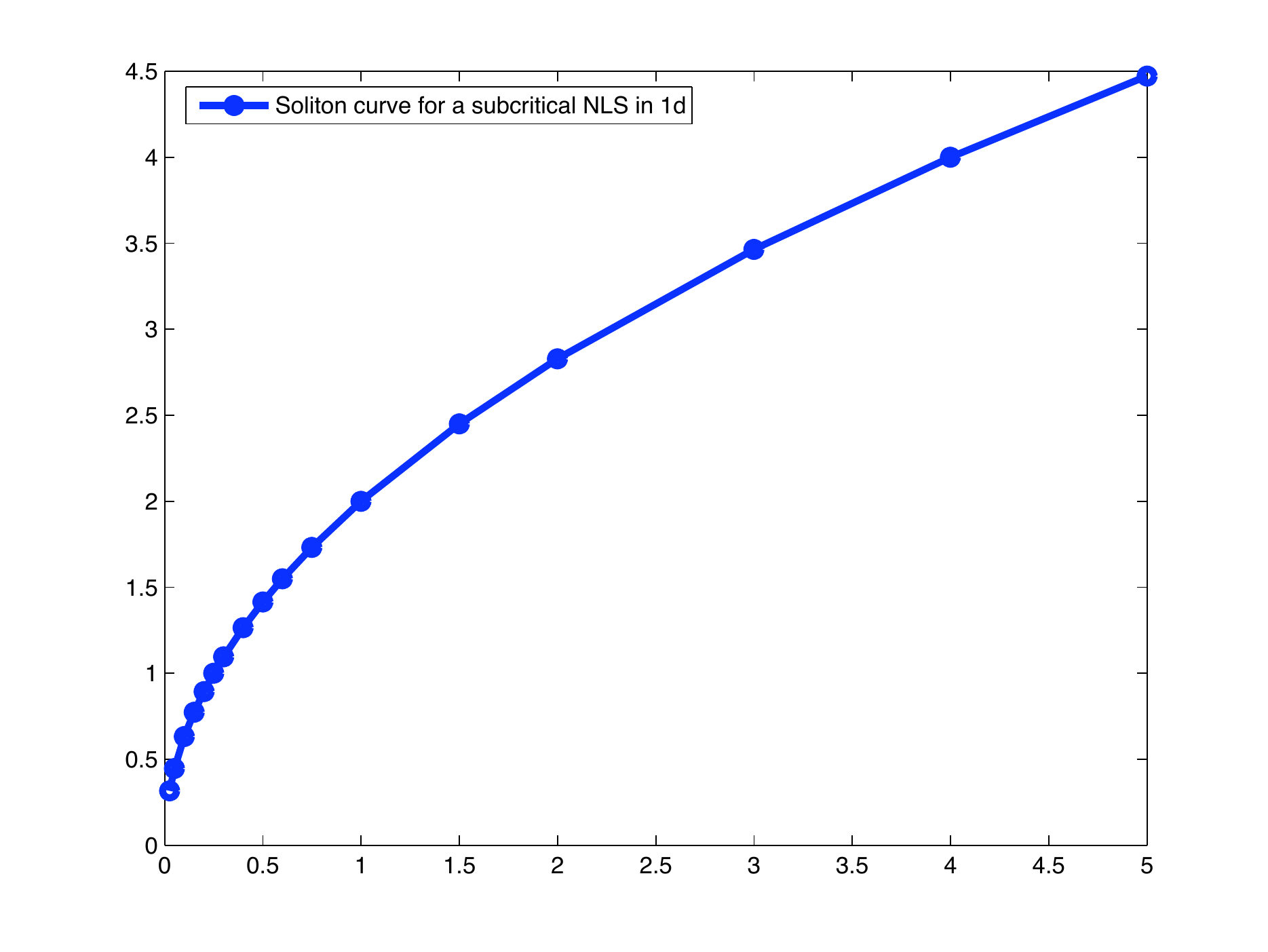}} \hfill \scalebox{0.38}{\includegraphics{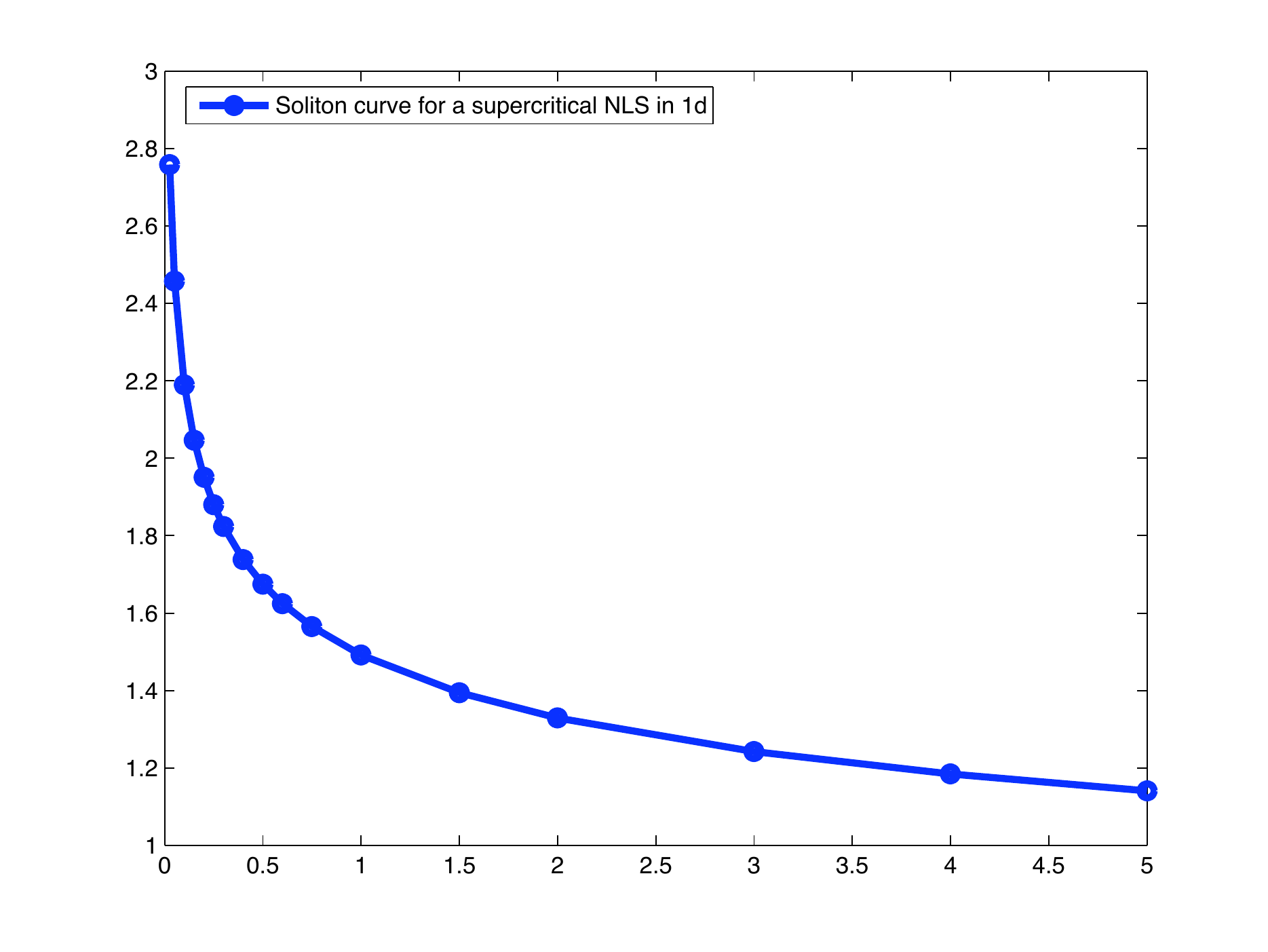}} 
\scalebox{0.38}{\includegraphics{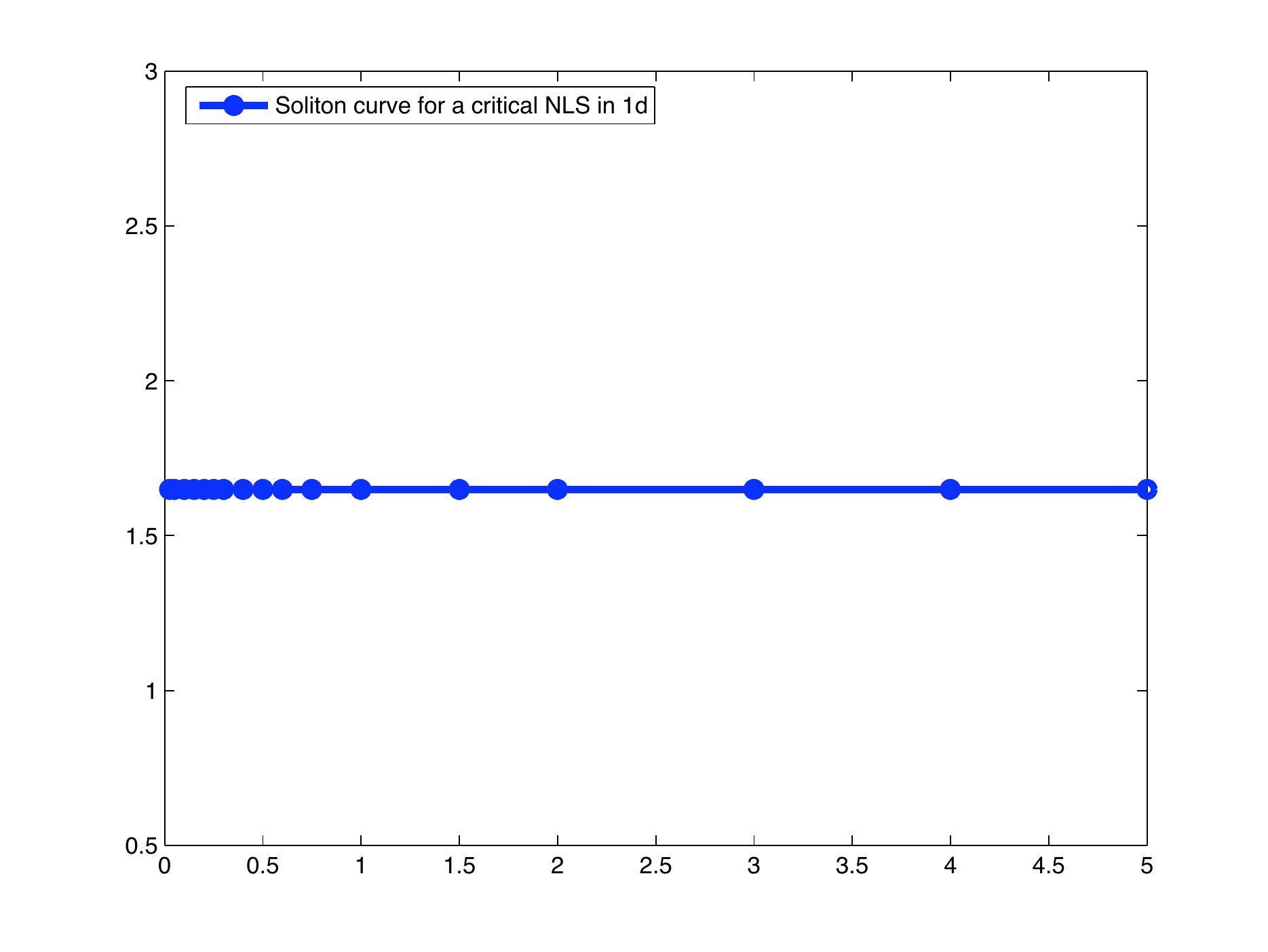}} \hfill \scalebox{0.38}{\includegraphics{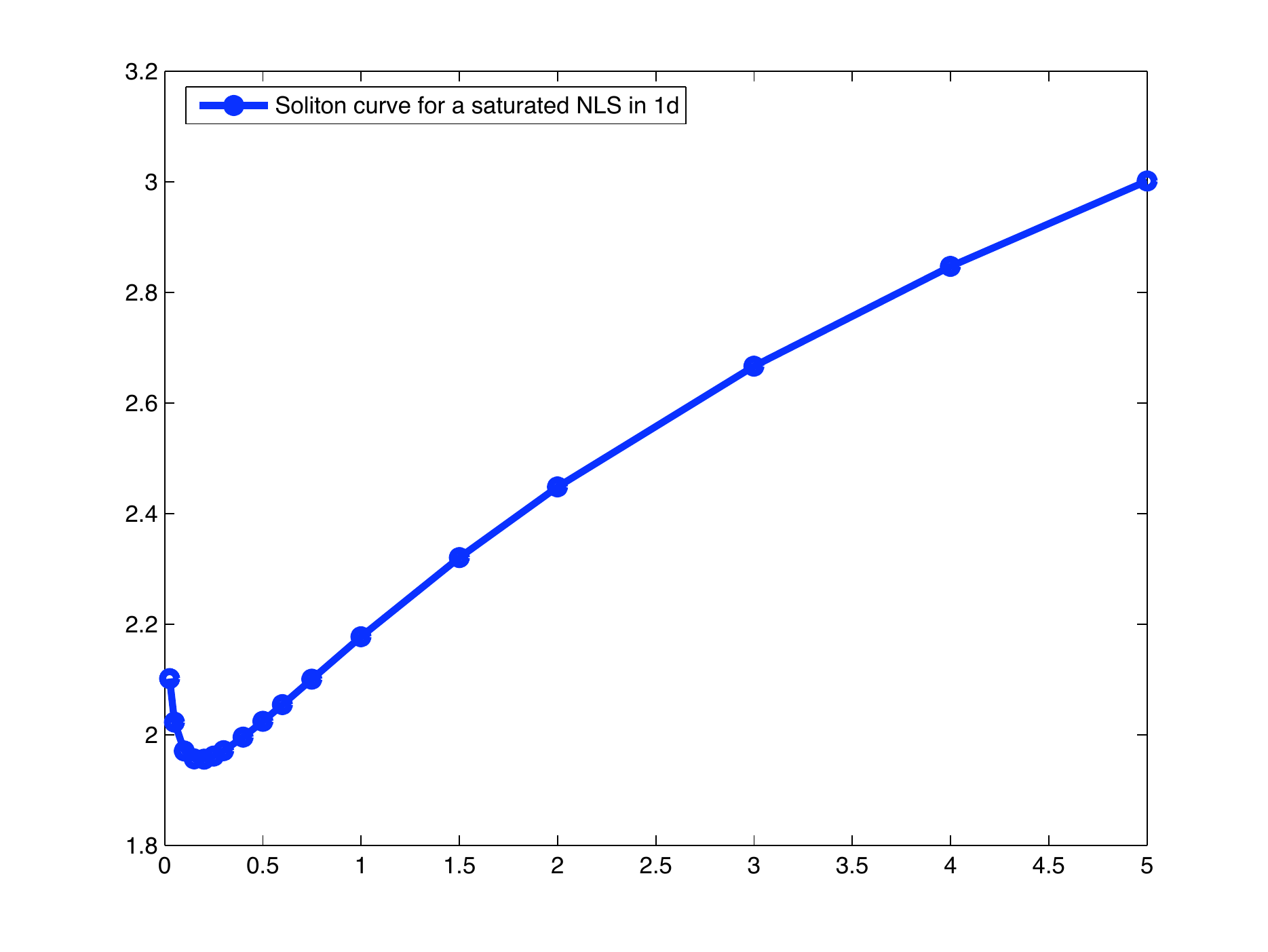}} 
\scalebox{0.38}{\includegraphics{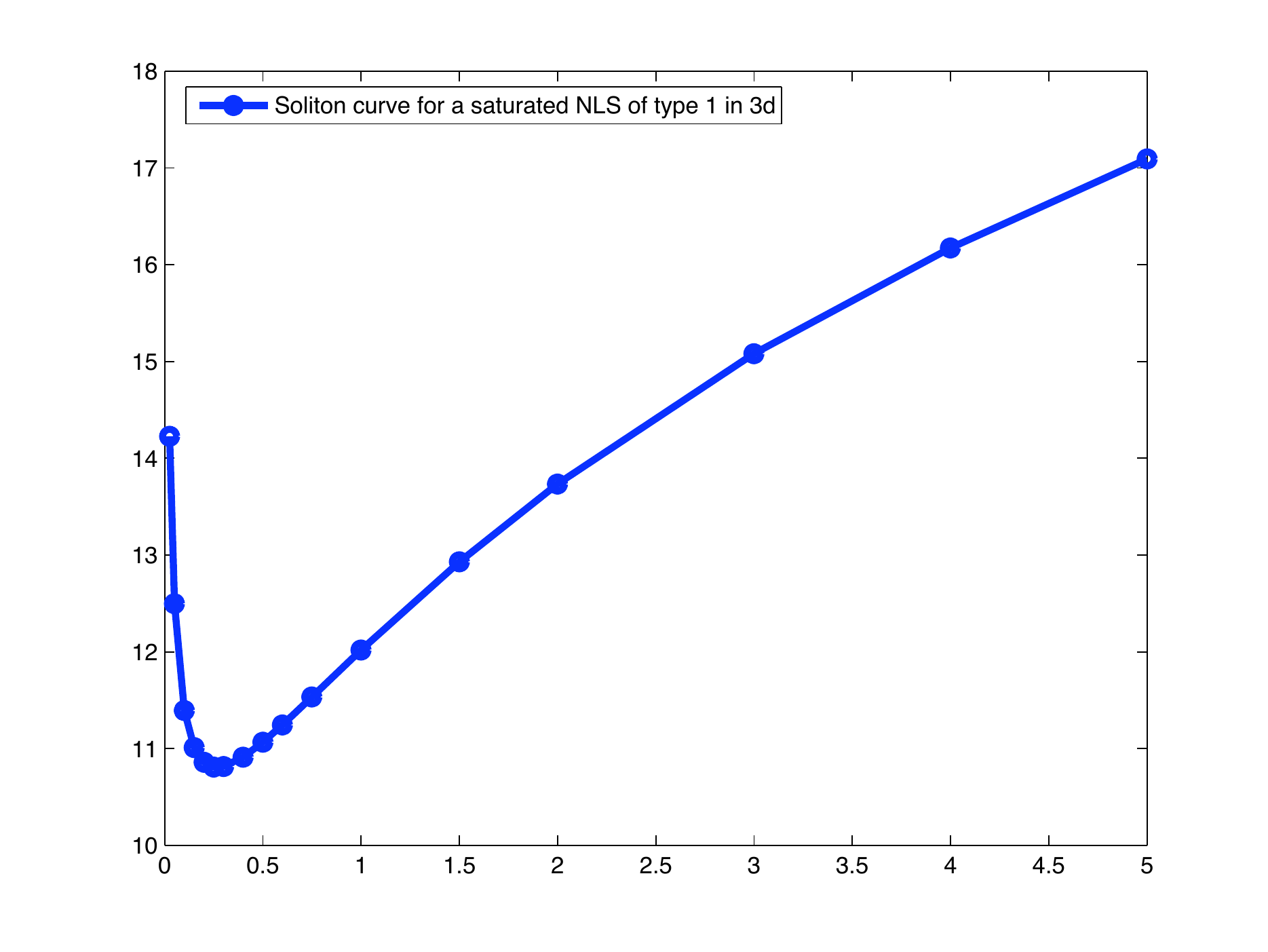}} \hfill \scalebox{0.38}{\includegraphics{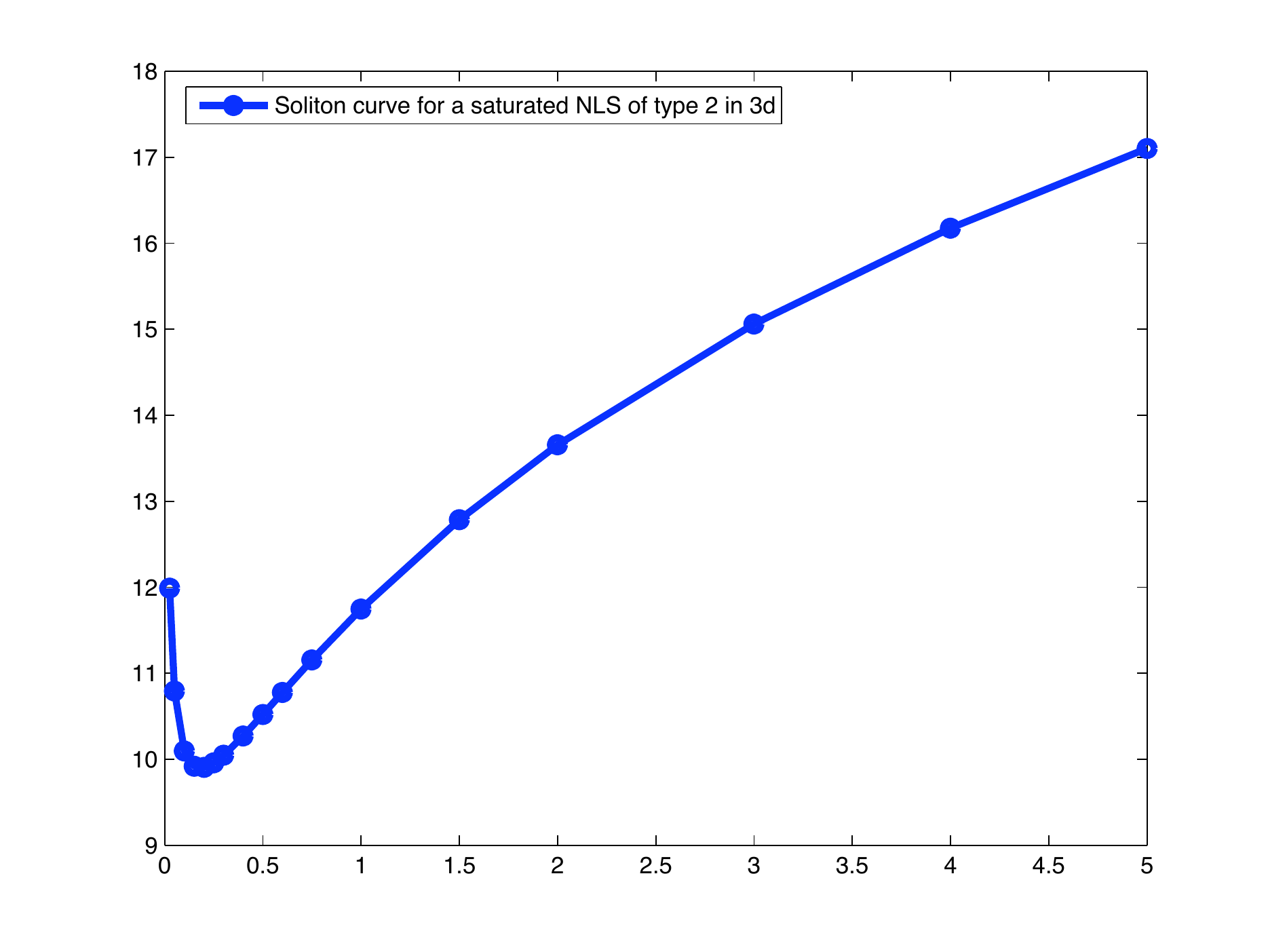}}

\caption{Plots of the soliton curves ($Q(\lambda)$ with respect to $\lambda$) for a subcritical nonlinearity ($d=1$, $p=3$), supercritical nonlinearity ($d=3$, $p=3$), critical nonlinearity ($d=1$, $p=5$), saturated nonlinearity of type $1$ ( $p=7$, $q=3$) in $\reals$, saturated nonlinearity of type $1$ in $3d$ ($p=4$, $q=2$), saturated nonlinearity of type $2$ in $\reals^3$ ($q=2$).  The curves for the monomial nonlinearities are found analytically, while the curves for the saturated nonlinearities are found numerically.}
\label{f:solcurves}
\end{figure}

\section{Linearization about a Soliton}

Let us write down the form of (NLS) linearized about a soliton solution.  First of all, we assume we have a solution $\psi = e^{i \lambda t}(R_{\lambda} + \phi(x,t))$.  For simplicity, set $R = R_\lambda$.  Inserting this into the equation we know that since $R$ is a soliton solution we have 
\begin{eqnarray}
i (\phi)_t + \Delta (\phi) & = & -\beta( R^2) \phi - 2 \beta'( R^2 )  R^2 \text{Re}(\phi) + O (\phi^2),
\end{eqnarray}
by splitting $\phi$ up into its real and imaginary parts then doing a Taylor Expansion. Hence, if $\phi = u + iv$, we get
\begin{eqnarray}
\partial_t \left( \begin{array}{c}
u \\
v
\end{array} \right) = \mathcal{H} \left( \begin{array}{c}
u \\
v
\end{array} \right),
\end{eqnarray}
where 
\begin{eqnarray}
\mathcal{H} =  \left( \begin{array}{cc}
0 & L_{-} \\
-L_{+} & 0
\end{array} \right),
\end{eqnarray}
where $$L_{-} = - \Delta + \lambda - \beta( R_\lambda )$$ and 
$$L_{+} = - \Delta + \lambda - \beta( R_\lambda ) - 2 \beta' (R^2_\lambda) R_\lambda^2.$$

\begin{defn}
\label{spec:defn1}
A Hamiltonian, $\mathcal{H}$, is called admissible if the following conditions hold: \\
1)  There are no embedded eigenvalues in the essential spectrum, \\
2)  The only real eigenvalue in $[-\lambda, \lambda ]$ is $0$, \\
3)  The values $\pm \lambda$ are not resonances. \\
\end{defn}

\begin{defn}
\label{spec:defn2}
Let (NLS) be taken with nonlinearity $\beta$.  We call $\beta$ admissible if there exists a minimal mass soliton, $R_{min}$, for (NLS) and the Hamiltonian, $\mathcal{H}$, resulting from linearization about $R_{min}$ is admissible in terms of Definition \ref{spec:defn1}.
\end{defn}

The spectral properties we need for the linearized Hamiltonian equation in order to prove stability results are precisely those from Definition \ref{spec:defn1}.  Notationally, we refer to $P_d$ and $P_c$ as the projections onto the discrete spectrum of $\mathcal{H}$ and onto the continuous spectrum of $\mathcal{H}$ respectively.  Analysis of these spectral conditions will be done both numerically and analytically in \cite{Mspec}.

\section{Main Results}
\label{int:main}

We derive the existence and important properties of distorted Fourier bases $\tilde{\phi}_{\xi}$ of non-self-adjoint matrix Hamiltonians, and hence a distorted Fourier transform, for a general class of matrix Hamiltonians.  Let $\mathcal{S}$ be the Schwartz class of functions.  Then, we have the following results:  

\begin{thm}
\label{thm:tdec}
Given an admissible Hamiltonian $\mathcal{H}$, and the projection on the continuous spectrum of $\mathcal{H}$, $P_c$, for initial data $\phi \in \mathcal{S}$, we have
\begin{eqnarray*}
\| e^{it\mathcal{H}} P_c \phi \|_{L^\infty} \leq t^{-\frac{d}{2}}.
\end{eqnarray*}
\end{thm}

Let us define the space 
\begin{eqnarray*}
L^{1,M} = \{ f \in L^1 | \| \langle \cdot \rangle^N f (\cdot) \|_{L^1} \leq \infty, \ N = 0,1,\dots,2M \},
\end{eqnarray*} 
with norm $\| \cdot \|_{L^{1,M}}$ defined in the standard fashion.  

\begin{thm}
\label{thm:wlindec}
Let $\mathcal{H}$ be an admissible Hamiltonian as defined above.  Assume $\vec{\psi} \in L^{1,M}$ and
\begin{eqnarray}
\label{eqn:mom1}
\partial^\alpha_\xi \partial^\beta_{|\xi|} \vec{\Psi} (0) = 0,
\end{eqnarray}
for multi-indices $\alpha$, $\beta$ such that $|\alpha| + |\beta| = 0,1,2,\dots,2M$, where
\begin{eqnarray*}
\vec{\Psi} (\xi) = \int_y \tilde{\phi}_{\xi} (y) \vec{\psi} (y) dy.
\end{eqnarray*}  
Then,
\begin{eqnarray}
\label{eqn:lindecw}
\| e^{-c |x|} e^{i t \mathcal{H}} P_{c} \vec{\psi} \|_{L^\infty} \leq C t^{-\frac{d}{2}-M} \| \vec{\psi} \|_{L^{1,M}},
\end{eqnarray}
for any $c > 0$.
\end{thm}

It should be noted that similar estimates were proven in the works \cite{ES1} and \cite{BW}, where in the first the techniques used were more along the lines of resolvent estimates and in the second the fact that the nonlinearities of interest were of even integer powers was crucial to the argument.  Here, we take an approach similar to that of scattering theory as presented in \cite{Ho2}.  Scattering theory is related to a resolvent approach most certainly, though there are certain benefits to the method we thought would be of general interest.  Note, these dispersive estimates are essential for the forthcoming argument in \cite{Mnl}, where perturbations of minimal mass solitons are analyzed.    

\section{General Distorted Fourier Basis Theory}
\label{lin:gdfb}

We present here a review of combined results from \cite{Ag} and \cite{Ho2}, Chapter 14.  Both presentations are valid for operators of the form
\begin{eqnarray*}
(P(D) + V(x,D)) u = 0,
\end{eqnarray*}
where $P(D)$ is a self-adjoint, constant coefficient differential operator and $V(x,D)$ is a short range, symmetric differential operator.  The perturbation $V(x,D)$ is defined to be short range in order to say that 
\begin{eqnarray*}
\lim_{z \to \lambda, \pm \Im z > 0} R(z) = R^{\pm} (z)
\end{eqnarray*}
exists in the uniform operator topology of $B(L^{2,s},\mathcal{H}_{2,-s})$, where
\begin{eqnarray*}
L^{2,s} (\reals^d) = \{ u(x) | (1+|x|^2)^{\frac{s}{2}} u(x) \in L^2 \}
\end{eqnarray*}
and
\begin{eqnarray*}
\mathcal{H}_{m,s} = \{ u(x) | D^\alpha u \in L^{2,s}, \ 0 \leq |\alpha| \leq m \}.
\end{eqnarray*}

Also, for any $f \in L^{2,s}$,
\begin{eqnarray*}
R^{\pm} (\lambda) f = R_0^\pm (\lambda ) f - R_0^\pm (\lambda ) V R^\pm (\lambda) f,
\end{eqnarray*}
where $R_0$ is the resolvent for the constant coefficient operator, $P$.  As the notion of short range deals with compactness of the operator $Z(u) = R (V u)$, being short range requires sufficient decay assumptions at $\infty$ on $V$.  Heuristically, it is required that the coefficients of $V$ decrease as fast as an integrable function in $|x|$ and for each fixed $x_0$, we have
\begin{eqnarray*}
\frac{V(x_0,\xi)}{P (x, \xi )} \to 0 \ \text{as} \ \xi \to \infty.
\end{eqnarray*}

The reasons why these heuristics hold true are explored below, hence we forgo this analysis here and move on with the fact that $V(x,D)$ is a short range perturbation as an assumption.  Note that in the case explored below, $V$ is Schwartz in $x$ and is dominated by $P(\xi)$ as $|\xi| \to \infty$.  It is also important to note that while contour integration works out nicely in $\RR^3$, the results presented here hold in any dimension where $R_0^+$ and $R_0^-$ are arrived at through a limiting procedure.

The Agmon approach to the distorted Fourier transform is equivalent to the approach taken by the author.  Namely, we define
\begin{eqnarray*}
\phi_{\pm} (x,\xi) = e^{ix \xi} - R^{\mp} (|\xi|) [ V e^{i x \cdot \xi}] (x).
\end{eqnarray*}
Then, the distorted Fourier transform is a map $\mathcal{F}_\pm :L^2 \to L^2$ such that \\
(i) $\text{Ker} (\mathcal{F}_\pm) = L^2_d$, where $L^2_d$ is the restriction of $L^2$ to the discrete spectrum of $P$.  Similarly, we have $L^2_c$, the restriction of $L^2$ to the continuous spectrum of $P$.  Then, the restriction of $\mathcal{F}_\pm$ is a unitary operator from $L^2_{c}$ onto $L^2$, \\
(ii) for any $ f \in L^2$
\begin{eqnarray*}
(\mathcal{F}_\pm f ) (\xi) = (2 \pi)^{-\frac{d}{2}} \lim_{N \to \infty} \int_{|x| < N} f(x) \overline{\phi_\pm (x,\xi)} dx \ \text{in $L^2_\xi$},
\end{eqnarray*}
and
\begin{eqnarray*}
(\mathcal{F^*}_\pm f ) (x) = (2 \pi)^{-\frac{d}{2}} \lim_{j \to \infty} \int_{K_j} f(\xi) \phi_\pm (x,\xi) d\xi \ \text{in $L^2_x$}
\end{eqnarray*}
where $K_j$ is an increasing sequence of compact sets such that $\cup_j K_j = \reals^d \setminus \mathcal{N}$ for 
\begin{eqnarray*}
\mathcal{N}(H) = \{ \xi \in \reals^d | |\xi|^2 \ \text{is an eigenvalue for $H$} \} \cup 0,
\end{eqnarray*}
and \\
(iii) If $P_{c}$ is the projection of $L^2$ onto $L^2_{c}$, then
\begin{eqnarray*}
(P_{c} H) f = (\mathcal{F}^*_\pm M_{P(\xi)} \mathcal{F}_\pm) f
\end{eqnarray*}
for any $f \in D(H)$ where $M_{P(\xi)}$ denotes multiplication by $P(\xi)$.  

In addition, we have $\| P_{c} f \|_{L^2} = \| \mathcal{F}_\pm f  \|_{L^2}$.  In other words, we have a Plancherel theorem for our distorted Fourier basis.  

Now, \cite{Ho2}, Chapter 14 arrives at the same conclusions using 
\begin{eqnarray*}
(\mathcal{F}_\pm f ) (\xi) = \mathcal{F} (I + V R_0^\pm)^{-1} f (\xi).
\end{eqnarray*}
However, using the resolvent identity
\begin{eqnarray*}
R(z) = R_0(z) (I + V R_0 (z))^{-1},
\end{eqnarray*}
we will see that a formal iteration shows equivalence between these definitions for $\xi$ large.  It is precisely this iteration we use below to get uniform bounds in $\xi$.  

\section{Convolution Kernels}
\label{s:ck}

In this section, we derive the integral kernel in $\RR^3$ for the inverse of the differential operator
\begin{eqnarray*}
P_{\mu} & = & -\Delta - |\xi_0|^2 \\
& = & -\Delta - \mu^2 ,
\end{eqnarray*}
where we have set $\mu = |\xi_0|$ for simplicity.
This will be quite useful in deriving the distorted Fourier basis functions for more complicated operators belows.

Specifically, given $u,f:\RR^3 \to \RR$, we find $K_\mu (x,y)$ such that if
\begin{eqnarray*}
P_{\mu} u = f,
\end{eqnarray*}
then
\begin{eqnarray*}
u = \int_{\RR^3} K_\mu (x,y) f(y) dy.
\end{eqnarray*}

To begin, we Fourier transform the equation to see
\begin{eqnarray*}
(\xi^2 - \mu^2) \hat{u} = \hat{f},
\end{eqnarray*}
hence
\begin{eqnarray*}
u = \mathcal{F}^{-1}[(\xi^2 - \mu^2)^{-1}]*f.
\end{eqnarray*}

So,
\begin{eqnarray*}
K_\mu (x,y) = \mathcal{F}^{-1}[(\xi^2 - \mu^2)^{-1}](x-y),
\end{eqnarray*}
if we can define 
\begin{eqnarray*}
G(x) = \mathcal{F}^{-1}[(\xi^2 - \mu^2)^{-1}](x)
\end{eqnarray*}
in a meaningful sense.  

Without loss of generality, set $\mu > 0$.  Initially, asssume that $x \neq 0$, though this will be easily seen as a limiting case in the end.  We have
\begin{eqnarray*}
G(x) & = & \int_{\RR^3} \frac{e^{ix \cdot \xi}}{(\xi^2 - \mu^2)} d\xi \\
& = & \int_0^\infty \int_0^{2\pi} \int_0^{\pi} \frac{e^{i |x| r cos(\theta)}}{(r-\mu)(r+\mu)} r^2 \sin (\theta) d\theta d\phi dr \\
& = & \frac{4 \pi}{|x|} \int_0^\infty \frac{r \sin(|x|r)}{(r-\mu)(r+\mu)} dr
\end{eqnarray*}
by first making a rotational change of variables where $\xi_3 \to \frac{x}{|x|}$, then using polar coordinates.  

Now, we are set up to use contour integration to find $G(x)$.  See Figure \ref{fig:fund1} for the contour over which we integrate.  We call this contour $\Gamma_{R,\epsilon}$.  

\begin{figure}
\scalebox{0.7}{\includegraphics{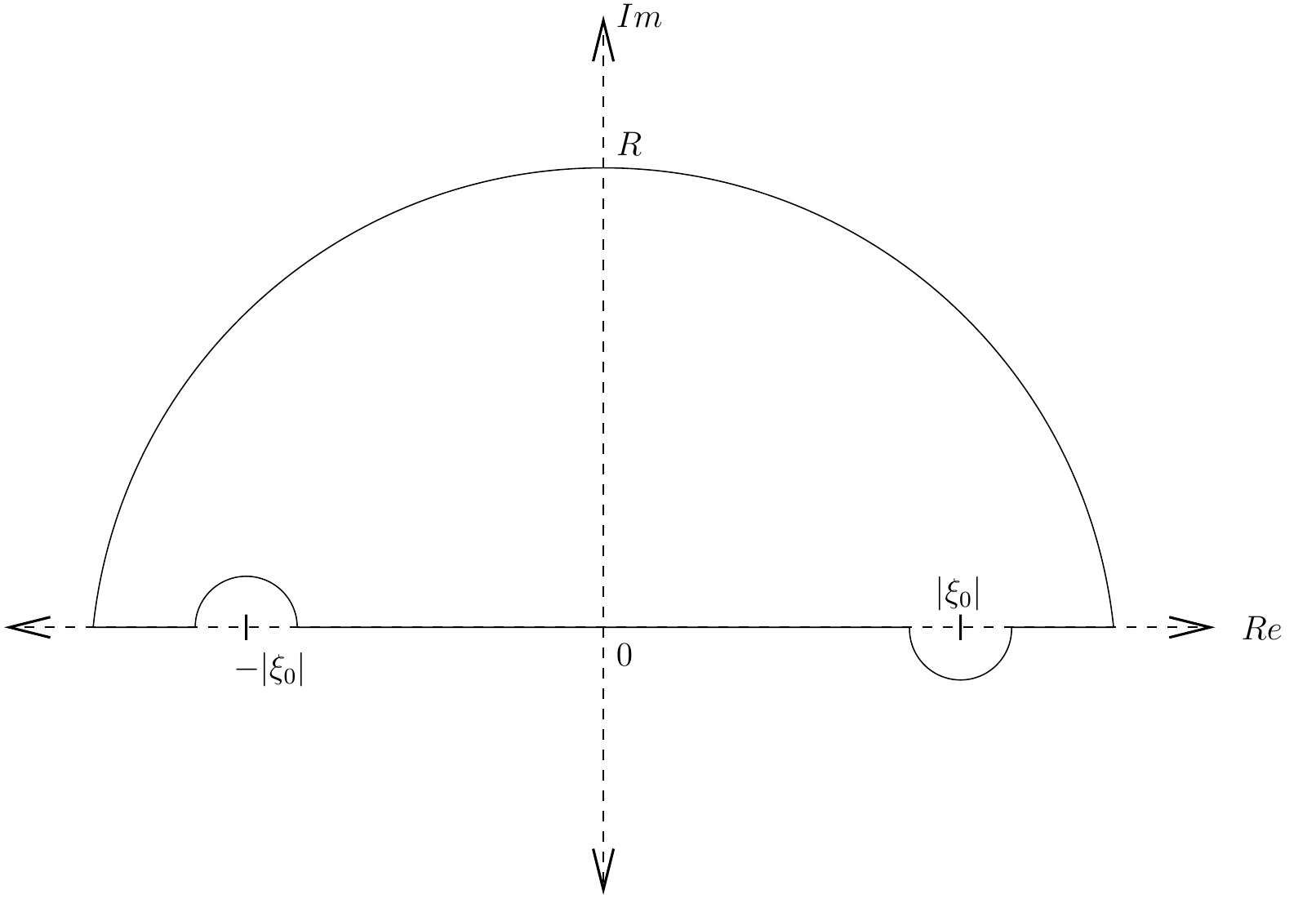}}
\caption{The contour for computing the behavior of $K_\mu$}
\label{fig:fund1}
\end{figure}

Then, we have from residue theory
\begin{eqnarray*}
\int_{\Gamma_{R,\epsilon}} \frac{z e^{iz|x|}}{(z-\mu)(z+\mu)} dz & = & 
2 \pi i \left[ \frac{\mu e^{i|x|\mu}}{2\mu} \right] \\
& = & \pi i e^{i|x|\mu}.
\end{eqnarray*}

However, breaking $\Gamma$ down, we also have
\begin{eqnarray*}
\int_{\Gamma_{R,\epsilon}} \frac{z e^{iz|x|}}{(z-\mu)(z+\mu)} dz = 2 i \int_0^R \frac{r \sin(|x|r)}{(r-\mu)(r+\mu)} dr + \frac{\pi i e^{i|x|\mu}}{2} - \frac{\pi i e^{-i|x|\mu}}{2}.
\end{eqnarray*}

Combining terms and taking $R \to \infty$, we have
\begin{eqnarray*}
G(x) = \frac{4 \pi^2 \cos(\mu |x|)}{|x|}.
\end{eqnarray*}
This is valid for all $x$ since the integral diverges as $x \to 0$.

Using simple residue theory, taking the distributional conventions
\begin{eqnarray*}
f( \lambda ) = f(\lambda + i0)
\end{eqnarray*}
or
\begin{eqnarray*}
f( \lambda ) = f(\lambda - i0)
\end{eqnarray*} 
result in 
\begin{eqnarray}
\label{eqn:G}
G^{\pm} (x) = \frac{4 \pi^2 e^{\pm i |x| \mu}}{|x|}.
\end{eqnarray}
To see this, define
\begin{eqnarray*}
G^+_\epsilon (x) & = & \mathcal{F}^{-1} [ (\xi^2 - (\mu+i\epsilon)^2)^{-1}] (x) \\
& = & \mathcal{F}^{-1} [ ( (|\xi| - \mu-i\epsilon) (|\xi| + \mu + i \epsilon) )^{-1}] (x).
\end{eqnarray*}
Now, we may make the same change of variables and do contour integration as above, though in this case we need not worry about avoiding $\pm |\xi_0|$.  So, our contour $\Gamma_{R,0}$ is the hemisphere on the upper half plane formed by the real axis and the half circle of radius say $R > \mu$.  The only residue in such a region would be given by $ z = \mu + i \epsilon$ as $z = \mu - i \epsilon$ is outside $\Gamma_{R,0}$.  For each $\epsilon$, we then have
\begin{eqnarray*}
G^+_\epsilon (x) = \frac{4\pi}{|x|} e^{i |x| \mu} e^{-|x| \epsilon}.
\end{eqnarray*}
Taking $\epsilon \to 0$ gives formula \eqref{eqn:G} for $G^+$.  The analysis for $G^-$ is similar.

The above analysis is then easily seen to be equivalent to applying to the distributional connvention
\begin{eqnarray*}
f( \lambda ) = \frac{1}{2} \left[ f(\lambda + i0) + f(\lambda - i0) \right],
\end{eqnarray*}
namely the case where both residues lying on the real axis must be taken into account.  However, since our eventual goal is to work with oscillatory integrals, for convenience and without loss of generality, we will work with the complex operator
$G (x) = G^+ (x)$.

\section{Distorted Fourier Basis}
\label{lin:dfb}

Note that in the sequel, we take the convention that the soliton parameter is $\lambda^2$ instead of $\lambda$.  This serves to remind the reader of the positivity of this parameter.  The convention of $\lambda$ slightly simplifies the variational formulation, but has no impact on the linear analysis presented here.

We seek to understand the functions in the continuous spectrum of $\mathcal{H}$ by decomposing them using a distorted Fourier basis given by 
\begin{eqnarray}
\label{eqn:dfbe}
(-\Delta + \lambda^2 - V_1)(-\Delta + \lambda^2 - V_2) u_{\xi_0} =  (\lambda^2 + |\xi_0|^2)^2 u_{\xi_0},
\end{eqnarray}
where $u_{\xi_0} = e^{i x \xi_0} + g_{\xi_0}$ and $g_{\xi_0}$ is yet to be determined.  

From \eqref{eqn:dfbe},
\begin{eqnarray*}
[(-\Delta + \lambda^2)^2-(\lambda^2 + |\xi_0|^2)^2] u_{\xi_0} & = & (-\Delta + \lambda^2) V_2 u_{\xi_0} + V_1 (-\Delta + \lambda^2 - V_2) u_{\xi_0}. 
\end{eqnarray*}
Hence,
\begin{eqnarray}
\label{eqn:dfbeq}
[(-\Delta + \lambda^2)^2-(\lambda^2 + |\xi_0|^2)^2] g_{\xi_0} & = & F_{\xi_0} (x) e^{i x \xi_0} + \tilde{V} (x,D) g_{\xi_0},
\end{eqnarray}
where
\begin{eqnarray*}
\tilde{V} (x,D) = V_1 (-\Delta + \lambda^2 - V_2),
\end{eqnarray*}
and $F_{\xi_0} (x)$ is a Schwartz function.  Then, taking the Fourier Transform, we have
\begin{eqnarray*}
[(|\xi|^2 + \lambda^2)^2 - (|\xi_0|^2+\lambda^2)^2] \hat{g}_{\xi_0} = 
\hat{F}(\xi;\xi_0) + ( \tilde{V}_{\mathcal{F}} \hat{g}_{\xi_0} ) (\xi),
\end{eqnarray*}
where
\begin{eqnarray*}
( \tilde{V}_{\mathcal{F}} g ) (\xi) & = & \lambda^2 (\hat{V}_2 + \hat{V}_1) * (g) + (|\xi|^2 \hat{V}_2) * (g) + (\hat{V}_2 + \hat{V}_1) * (|\xi|^2 g) \\
& + & (\xi \hat{V}_2) * ( \xi g) - (\widehat{V_1 V_2})*(g).
\end{eqnarray*}

Given
\begin{eqnarray*}
L_{\xi_0} & = &  [(|\xi|^2 + \lambda^2)^2 - (|\xi_0|^2+\lambda^2)^2] \\
& = & [(|\xi|+|\xi_0|)(|\xi| - |\xi_0|)(|\xi|^2 + 2 \lambda^2 + |\xi_0|^2)],
\end{eqnarray*}
we have
\begin{eqnarray*}
g_{\xi_0} & = & \mathcal{F}^{-1} \left\{ \{ L_{\xi_0}^{\pm} \}^{-1} ( \hat{F} + \tilde{V}_{\mathcal{F}} \hat{g}_{\xi_0}) \right\} \\
& = & K^{\pm}_{\xi_0} * F_{\xi_0} + K^{\pm}_{\xi_0} * (\tilde{V} (x,D) g_{\xi_0}),
\end{eqnarray*}
where 
\begin{eqnarray*}
K^{\pm}_{\xi_0} (x)  = (\mathcal{F}^{-1} \{ L_{\xi_0}^{\pm} \}^{-1}) (x)
\end{eqnarray*}
and
\begin{eqnarray*}
L_{\xi_0}^{\pm} = [(|\xi|+|\xi_0| \pm i0)(|\xi| - |\xi_0| \mp i0)(|\xi|^2 + 2 \lambda^2 + |\xi_0|^2)].
\end{eqnarray*}
Note that for simplicity we have omitted a small complex perturbation in the elliptic term $(|\xi|^2 + 2 \lambda^2 + |\xi_0|^2)$ since it does not effect the analysis.

To explore $K^{\pm}_{\xi_0}$ further, we see in $\R^3$
\begin{eqnarray*}
\int_{\xi} \frac{e^{i \xi \cdot x}}{(|\xi|+|\xi_0| \pm i0)(|\xi| - |\xi_0| \mp i0)(|\xi|^2 + 2 \lambda^2 + |\xi_0|^2)} d\xi = \\ 
\int_{\R^3}  \frac{e^{i \xi_1 |x|}}{(|\xi|+|\xi_0| \pm i0 )(|\xi| - |\xi_0| \mp i0)(|\xi|+|\xi_0|)(|\xi|^2 + 2 \lambda^2 + |\xi_0|^2)} d\xi,
\end{eqnarray*}
using the change of variables $\xi_1 \to \frac{x}{|x|}$.  Then, we have
\begin{eqnarray*}
\int^{2 \pi}_0 \int^\pi_0 \int_0^R \frac{e^{r \cos(\theta) |x|}}{(r + |\xi_0| \pm i0) (r - |\xi_0| \mp i0)(r^2 + 2 \lambda^2 + |\xi_0|^2)} r \sin(\theta) dr d\theta d\phi.
\end{eqnarray*}
Doing integration first in $\theta$, then a contour integral, we have as in Section \ref{s:ck} that
\begin{eqnarray*}
\label{eqn:K}
K^{\pm}_{\xi_0} = \hat{L^{-1}_{\xi_0}} = \frac{\pi^2}{|\xi_0|^2+\lambda^2} \left[\frac{e^{\pm i |x| |\xi_0|} - e^{-|x|\sqrt{|\xi_0|^2+2 \lambda^2}}}{|x|}\right].
\end{eqnarray*}
For simplicity, we take $K(x) = K^{+}_{\xi_0} (x)$ as the analysis for $K^{-}_{\xi_0}$ will be similar.
Then, we want to use an iterative argument to show that for mid to high range frequencies, these distorted Fourier bases exist in $L^4$.  It will become clear in the sequel why $L^4$ is chosen. Note that since near $0$, $K$ is bounded, we have $K \in L^{3+s}$ for any $s>0$.  In particular we show the following:

\begin{figure}
\includegraphics{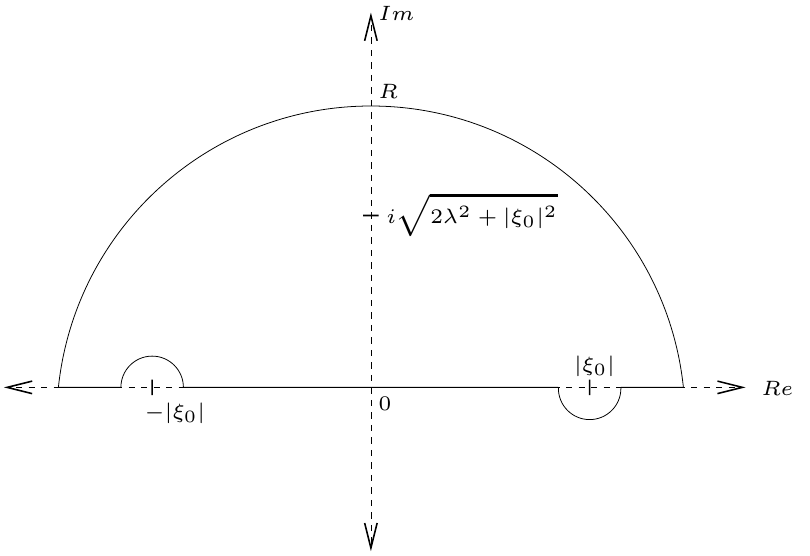}
\caption{The contour for computing the behavior of the fundamental solution in the limiting case.}
\label{fig:fund}
\end{figure}

\begin{lem}
\label{lem:K}
For the operator $K^{\pm}$ defined in Equation \eqref{eqn:K}, we have
\begin{eqnarray*}
K^{\pm} : L^{\frac{4}{3}} \to L^4 \ (O((|\xi_0|^2 + \lambda^2)^{-1} |\xi_0|^{-\frac{1}{2}})). 
\end{eqnarray*}
\end{lem}

\begin{proof}
We actually prove the result for
\begin{eqnarray*}
\tilde{K}^a_{\xi_0} f (x) & = & \mathcal{F}^{-1} \left( \frac{1}{(|\xi|^2 - (|\xi_0| + i0)^2 )^a} f \right) \\
& = & \int k_a (x,y) f(y) dy.
\end{eqnarray*}
The proof for $K$ will be essentially the same.

Using distribution theory, we have for $s \in \RR$
\begin{eqnarray*}
\tilde{K}^{0}_{\xi_0} (x) & = & \delta (x),  \\
\tilde{K}^{1}_{\xi_0} (x) & = & \frac{4 \pi^2}{|x|} [e^{i |x||\xi_0|}], \\
\tilde{K}^{2}_{\xi_0} (x) & = & \frac{i 2 \pi^2 e^{i |x| |\xi_0|}}{ |\xi_0|}.
\end{eqnarray*}
As convolution operators, 
\begin{eqnarray*}
\tilde{K}^0 : L^2 \to L^2 \ (O(1))
\end{eqnarray*}
and
\begin{eqnarray*}
\tilde{K}^2 : L^1 \to L^\infty \ (O(|\xi_0|^{-1}))  ,
\end{eqnarray*}
hence we wish to wish to define $K^s$ in such a way as to preserve these estimates and such that $k_s$ is analytic for $0 < \text{Re} (s) < 2$ and continuous for $0 \leq \text{Re} (s) \leq 2$.  However, after making a branch cut on the left half of the real axis, for $s \in \RR$ we have
\begin{eqnarray*}
\| (|\xi|+|\xi_0|)^{-is} (|\xi| - |\xi_0|)^{-is} f(\xi) \|_{L^\infty_\xi} \lesssim \| f \|_{L^\infty_\xi},
\end{eqnarray*}
and continuity on $0 \leq \text{Re} (s) \leq 2$ follows easily on a strip in the complex plane.  For analyticity inside the strip, it is clear any factors gained taking derivatives will be logarithmic and hence controlled by the polynomially decaying coefficients from $\text{Re} (s)$.  Hence, using complex interpolation
\begin{eqnarray*}
\tilde{K}^1 : L^{\frac{4}{3}} \to L^4 \ (O(|\xi_0|^{-\frac{1}{2}})). 
\end{eqnarray*}

\end{proof}

For simplicity, we from now on write $\tilde{K}$ instead of $\tilde{K}_{\xi_0}^1$.  Now, we seek to analyze the equation
\begin{eqnarray}
\label{eqn:dfbu}
g_{\xi_0} = K^{\pm}_{\xi_0} * F_{\xi_0} + K^{\pm}_{\xi_0} * (\tilde{V} (x,D) g_{\xi_0}),
\end{eqnarray}
In particular, we have the following:

\begin{thm}
\label{thm:dfbu}
Let $P(x,D)$ be a differential operator of the form
\begin{eqnarray*}
P(x,D) = (-\Delta + \lambda^2 - V_1)(-\Delta + \lambda^2 - V_2),
\end{eqnarray*}
where $V_1$, $V_2 \in \mathcal{S}$.  Assuming that there are no eigenvalues embedded in the continuous spectrum $[\lambda^4,\infty)$, there exists $g^{\pm}_{\xi_0} \in L^4$ such that Equation \eqref{eqn:dfbu} is satisfied for $u_{\xi_0} = e^{i x \xi_0} + g^{\pm}_{\xi_0} (x)$.  We have 
\begin{eqnarray*}
g^{\pm}_{\xi_0} (x) = K^{\pm} * [f_0 (\cdot, \xi_0, |\xi_0|)],
\end{eqnarray*}
where $f_0$ is smooth in $x$, $\xi_0$, $|\xi_0|$, and
\begin{eqnarray*}
| \langle x \rangle^N \partial_x^\alpha f_0| \lesssim 1.
\end{eqnarray*}  
Moreover, there exists a value $M$ such that 
for $\xi_0 \geq M$, 
\begin{eqnarray*}
f_0(x,\xi_0) = e^{i (x, \xi_0)} f (x,\xi_0),
\end{eqnarray*}
where 
\begin{eqnarray}
\label{eqn:symb}
|\langle x \rangle^{N} \partial^\alpha_{\xi_0} \partial^\beta_x f (x, \xi_0) | \lesssim |\xi_0|^{2-|\alpha|},
\end{eqnarray}
for any multi-indices $\alpha$ and $\beta$, $N > 0$.
\end{thm}

\begin{proof}
The solution to \eqref{eqn:dfbu} will be solved differently for large and small values of $\xi_0$.  In particular, we use a Fredholm theory approach for the small frequencies and an iterative approach for the large frequencies.  The analysis will be done using $K^+$ as the analysis for $K^-$ will follow similarly.  For simplicity, we set $K = K^+$. 

To begin, let us take $|\xi_0|>M$, where $M$ will be determined in the exposition.  Then, we solve Equation \eqref{eqn:dfbu} using Picard iteration.  For simplicity, let $g_{\xi_0} = v$.  Setting $v^0 = 0$ and $T u = \tilde{V} (K * u)$, we have
\begin{eqnarray*}
v^1 & = & K(x)*[F_{\xi_0} (x) e^{i x \cdot \xi_0}] \\
v^2 & = & K(x) * [( F_{\xi_0} (x) e^{i x \cdot \xi_0}) + (\tilde{V}(x,D) K(x)*(F_{\xi_0} (x) e^{i x \cdot \xi_0}))] \\
& = & K(x) * [( F_{\xi_0} (x) e^{i x \cdot \xi_0}) - (V_1+V_2)(\lambda^2 + |\xi_0|^2)K(x)*(F_{\xi_0} (x) e^{i x \cdot \xi_0}) \\
&+& (V_1+V_2)\tilde{K}(x)*(F_{\xi_0} (x) e^{i x \cdot \xi_0}) - (\nabla V_2 \cdot \nabla K(x)*(F_{\xi_0} (x) e^{i x \cdot \xi_0})) \\
&-& (V_1 (x) V_2(x) + \Delta V_2) K(x)*(F_{\xi_0} (x) e^{i x \cdot \xi_0})]\\
&\vdots& \\
v^n & = & K(x) * [F_{\xi_0} (x) e^{ix \cdot \xi_0} + \tilde{V} (x,D) v^{n-1} ] \\
& = & K(x) * [ \sum_{m=0}^{n-1} T^{m} F_{\xi_0} (x) e^{ix \cdot \xi_0}  ] \\
&\vdots& .
\end{eqnarray*}

We wish to show that this iteration converges in $L^4$.  
To see this, let $u \in L^4$.  Note that
\begin{eqnarray*}
\| K * \tilde{V} (x,D) u \|_{L^4} \lesssim \| K * V u \|_{L^4} + \| \nabla K * \bar{V} u \|_{L^4} + \| \Delta K * \bar{\bar{V}} u \|_{L^4},
\end{eqnarray*}
where $V$, $\bar{V}$, $\bar{\bar{V}} \in \mathcal{S}$.  Then,
\begin{eqnarray*}
\| K * \tilde{V} (x,D) u \|_{L^4} & \lesssim & \frac{1}{\xi_0^2 + \lambda^2} \frac{1}{|\xi_0|^{\frac{1}{2}}} \| V u \|_{L^\frac{4}{3}} + \frac{ |\xi_0| }{\xi_0^2 + \lambda^2} \| (|y|^{-1}) * \bar{V} u \|_{L^4} \\
& + & \frac{ \xi_0^2 }{\xi_0^2 + \lambda^2} \| K * \bar{\bar{V}} u \|_{L^4},
\end{eqnarray*}
so using the Hardy-Littlewood-Sobolev inequality and the bounds on $K$, we have
\begin{eqnarray*}
\| K * \tilde{V} (x,D) u \|_{L^4} & \lesssim & |\xi_0|^{-\frac{1}{2}} \| V u \|_{L^\frac{4}{3}} \\
& \lesssim & |\xi_0|^{-\frac{1}{2}} \| V \|_{L^2} \| u \|_{L^4},
\end{eqnarray*}
for some $V \in \mathcal{S}$.  As a result,
\begin{eqnarray*}
\| K * \tilde{V} (x,D) \|_{L^4 \to L^4} \leq C |\xi_0|^{-\frac{1}{2}},
\end{eqnarray*}
where $C$ is determined by $V_1$, $V_2$.  If $|\xi_0| > C^2$, then
\begin{eqnarray*}
\| K * \tilde{V} (x,D) \|_{L^4 \to L^4} \leq 1,
\end{eqnarray*}
and the existence of $g_\xi \in L^4$ for
\begin{eqnarray*}
(I - K * \tilde{V} (x,D)) g_\xi = g_\xi
\end{eqnarray*}
follows from a contraction argument.  In the notation from the theorem, we have $C^2 = M$.

Now, for the smaller frequencies, we apply Fredholm theory.  This approach also works for large $|\xi_0|$, however the iterative approach gives us uniform bounds for all $\xi_0$ such that $|\xi_0| > M $.  Once differentiability in $\xi_0$ has been obtained, we will then have uniform bounds for all $\xi_0$.  However, we must be careful near $\xi_0 = 0$ as $K$ has a particularly challenging dependence upon $|\xi_0|$.  We explore this shortly, but first let us finish the existence argument for low frequencies.  

To begin, Equation \eqref{eqn:dfbu} shows that
\begin{eqnarray}
\label{eqn:dfbg}
g_{\xi_0} = K*(\tilde{V} (x,\xi_0) e^{ix \cdot \xi_0}) + K*(\tilde{V}(x,D) g_{\xi_0}),
\end{eqnarray}
where 
\begin{eqnarray*}
\tilde{V} (x,D)  = (-\Delta + \lambda^2 - V_1) V_2 + V_1 (-\Delta + \lambda^2)
\end{eqnarray*}
is a second order operator.  

Now, if $K*( \tilde{V} (x,D) \cdot)$ is a compact operator, we may use Fredholm Theory (see \cite{LCE}, Appendix F) to say that either there is a unique solution to \eqref{eqn:dfbg} or there exists a nontrivial $u \in L^4$ such that
\begin{eqnarray*}
(I - K * \tilde{V}) u = 0.
\end{eqnarray*}
However, expanding the equation for $u$, we see this $u$ is an embedded resonance and hence an embedded eigenvalue from \cite{ES1} or \cite{Mspec}.  As our spectral assumptions preclude the existence of embedded eigenvalues, the solution to \eqref{eqn:dfbg} is unique.  

Let us now discuss the compactness.  The operator itself is of the form
\begin{eqnarray*}
K*(\tilde{V} v) & = & \int \pi^2 \frac{[ e^{i |x-y||\xi_0|} - e^{-|x-y|\sqrt{|\xi_0|^2+2\lambda^2}}]}{|x-y| (|\xi_0|^2 + \lambda^2)} \tilde{V} (y, D_y) v(y) dy \\
 & = & \int \pi^2 \frac{[ e^{i |x-y||\xi_0|} - e^{-|x-y|\sqrt{|\xi_0|^2+2\lambda^2}}]}{|x-y| (|\xi_0|^2 + \lambda^2)} \\
& \times & [(-\Delta_y + \lambda^2 - V_1(y)) V_2(y) + V_1(y) (-\Delta_y + \lambda^2)] v(y) dy.
\end{eqnarray*}

Hence, using integration by parts, we are concerned about the following two types of operators
\begin{eqnarray*}
(1) \ P_1 u & = & \int K(x-y) V(y) u(y) dy \\
(2) \ P_2 u & = & \int \tilde{K}(x-y) V(y) u(y) dy,
\end{eqnarray*}
where $V \in \mathcal{S}$.  Of course, technically there will be terms with derivatives falling on $K$ and $V$, however a brief calculation shows that these fall into the same class of operators as $P_2$.  Indeed, by construction
\begin{eqnarray*}
(-\Delta - | \xi_0|^2) \tilde{K} = 0
\end{eqnarray*}
and
\begin{eqnarray*}
(-\Delta - | \xi_0|^2) K = \frac{4 \pi^2}{|x|} [ e^{-|x|\sqrt{|\xi_0|^2+2 \lambda^2}}],
\end{eqnarray*}
hence when all derivatives fall on $K$, simply by looking at $-\Delta - | \xi_0|^2 + | \xi_0|^2$ we get reduction back to $P_1$ or $P_2$ as $K$ is a convolution kernel for an exact solution.  

We now need to prove 
$$P_i :L^4 \to L^4,$$  
for $i = 1,2$.

Assume that $u_j \to^w 0$ in $L^4$.  Since we are working in $\RR^3$, 
using duality and the properties of $V$, we have 
\begin{eqnarray*}
P_i u_j (x) \to 0 \ \text{as} \ j \to \infty
\end{eqnarray*}
for almost every $x$, where $i = 1,2$.  By the uniform boundedness of weakly convergent sequences, the Hardy-Littlewood-Sobolev Inequality, and H\"older we have,
\begin{eqnarray*}
\| P_i u_j \|_{L^4} & \leq & \| V \|_{L^\frac{3}{2}} \| u_j \|_{L^4} \\
& \leq & C,
\end{eqnarray*}
for $i = 1,2$.  Hence, there is a subsequence $j_k$ such that $\| P_i u_{j_k} \|_{L^4}$ converges.  Therefore, it must converge to $0$.  As a result, the operator $K*(\tilde{V} \cdot):L^4 \to L^4$ is compact and there exists a unique $g_{\xi_0}$ for all $\xi_0$.  Note that $\tilde{V} K$ is compact from $L^{\frac{4}{3}} \to L^{\frac{4}{3}}$ using similar arguments.  

To discuss the continuous dependence upon $\xi_0$, we need to study the functions $g_{\xi_0}$ in more detail.  In particular, we must have $\tilde{V} g_{\xi_0}$ smooth with repect to $\xi_0$ and $|\xi_0$.  From the expression for $g_{\xi_0}$, we know that
\begin{eqnarray*}
(I - K*(\tilde{V} (x,D) \cdot )) g_{\xi_0} & = & (I - P) g_{\xi_0} \\
& = & K * \tilde{V} (x,\xi_0) e^{i x \xi_0},
\end{eqnarray*}
so
\begin{eqnarray*}
g_{\xi_0} = (I - P)^{-1} (K * (\tilde{V} (x,\xi_0) e^{i \cdot \xi_0})),
\end{eqnarray*}
where
\begin{eqnarray*}
K (\xi) = [(-\Delta - \xi^2)(-\Delta + 2 \lambda^2 + \xi^2)]^{-1}.
\end{eqnarray*}
From Fredholm Theory and the spectral assumptions, $(I - P)^{-1}$ is a resolvent which is uniquely defined.  However, using the decay of $\tilde{V}$, we can write 
\begin{eqnarray*}
\tilde{V} = \tilde{V}_1 \tilde{V}_2,
\end{eqnarray*}
where $|e^{c |x|} \tilde{V}_1| \lesssim 1$, $|e^{c |x|} \tilde{V}_2 f| \lesssim \| f \|_{W^{2,\infty}}$ given $0<c<c_0$.  The constant $c_0$ is determined by the decay of $\tilde{V}$.  Hence, using a resolvent identity, we have
\begin{eqnarray*}
\tilde{V} g_{\xi_0} = \tilde{V}_1 (I - \tilde{V}_2 K \tilde{V}_1)^{-1} \tilde{V}_2 (K * (\tilde{V} (x,\xi_0) e^{i \cdot \xi_0})).
\end{eqnarray*}
Using the decay properties of $\tilde{V}_i$ for $i=1,2$ and the differentiability of $K$, for any $\xi_0$ we have $\tilde{V}_2 K \tilde{V}_1 (z) $ is well-defined for $z \in \mathbb{C}$ in a small neighborhood of $|\xi_0|$.  As a result, 
\begin{eqnarray*}
(I - \tilde{V}_2 K \tilde{V}_1)^{-1}
\end{eqnarray*}
is analytic with respect to $z$.  
Also, $K$ is analytic with respect to $|\xi|$ and $\xi$, $\tilde{V}_2 e^{i x \xi}$ is analytic with respect to $\xi$ and we see that $g_{\xi_0}$ depends smoothly on $|\xi|$ and $\xi$.  Using the resolvent identity
\begin{eqnarray*}
f_0 (x, \xi) = \tilde{V} e^{i x \cdot \xi} + \tilde{V} (1 - K \tilde{V})^{-1} K*(\tilde{V} e^{i x \cdot \xi}),
\end{eqnarray*}
the decay in $x$ for $f_0$ follows.  


For $|\xi_0| \geq M$, let us return to the iteration scheme
\begin{eqnarray*}
g^0_{\xi_0} & = & K * [\tilde{V} (\cdot , \xi_0) e^{i (\cdot, \xi_0)}], \\
& \vdots & \\
g^n_{\xi_0} & = & K * [\tilde{V} (\cdot , \xi_0) e^{i (\cdot, \xi_0)} + \tilde{V} (\cdot , \xi_0) g^{n-1}_{\xi_0} ],
\end{eqnarray*}
for $n \geq 1$.  Assuming $g_\xi = e^{i x \cdot \xi_0} f_0 (x, \xi_0, |\xi_0|)$, we have
\begin{eqnarray*}
f_0 & = & \tilde{V} (x,\xi_0) + e^{-i x \xi_0} \tilde{V} K*( e^{i x \xi_0} f_0) ,
\end{eqnarray*}
where by the mapping properties of $K$, choosing $M$ large enough, this expression is valid in $L^{\frac{4}{3}}_x$ for all $|\xi_0 | \geq M$.  

We would like to better understand the regularity in $x$ and $\xi$.  
To begin, let 
$$u = K* [e^{i (\cdot, \xi_0)} \phi(\cdot, \xi_0)].$$
Then, we see
\begin{eqnarray*}
(\partial_x - i \xi_0) u (x) & = & (\partial_x - i \xi_0) ( K * [\phi (\cdot,\xi_0) e^{i (\cdot, \xi_0)}] ) (x) \\
& = & i \xi_0 \int K(y) e^{i (x-y) \xi_0} \phi (x-y,\xi_0) dy - i \xi_0 \int K(y) e^{i (x-y) \xi_0} \phi (x-y,\xi_0)dy \\
& + & \int K(y) e^{i (x-y) \xi_0} \phi_x (x-y,\xi_0) dy \\
& = & \int K(y) e^{i (x-y) \xi_0} \phi_x (x-y,\xi_0) dy.
\end{eqnarray*} 
From here, recognizing that $e^{-i x \xi_0}$ cancels from
\begin{eqnarray*}
e^{-i x \xi_0} \tilde{V} K*( e^{i x \xi_0} \cdot)
\end{eqnarray*}
and again using the mapping properties of $K$, we have 
\begin{eqnarray*}
\| \partial^\alpha_x f_0 \|_{L^{\frac43}_x} \leq C_\alpha,
\end{eqnarray*}
for all multi-indices $\alpha$.  Hence, $f_0 \in C^\infty_x \cap L^\infty_x$. 
Similarly, 
\begin{eqnarray*}
\| \langle x \rangle^{N} \partial^\alpha_x f_0 \|_{L^{\frac43}_x} \leq C_{N,\alpha},
\end{eqnarray*}
for any $N \geq 0$ using the decay in $x$ of the operator $\tilde{V}$.

For the regularity in $\xi$, note that taking once again $u = K* [e^{i (\cdot, \xi_0)} \phi(\cdot, \xi_0)]$, we have
\begin{eqnarray*}
(\partial_{\xi_0} - i x) u & = & (\partial_{\xi_0} - i x) ( K * [\phi (\cdot,\xi_0) e^{i (\cdot, \xi_0)}] ) (x) \\
& = & \frac{4 \pi^2}{(\xi_0^2 + \lambda^2)} \left( i \frac{\xi_0}{|\xi_0|} \right) \int e^{i |x-y| |\xi_0|} e^{i (y) \xi_0} \phi (y,\xi_0) dy \\
& + & i \frac{\xi_0}{\sqrt{\xi_0^2 + 2 \lambda^2}}  \int e^{- |x-y| \sqrt{\xi_0^2 + 2 \lambda^2}} e^{i (y) \xi_0} \phi (y,\xi_0) dy  \\
& - & i x \int K(y) e^{i (x-y) \xi_0} \phi (x-y) dy + i \int K(y) e^{i (y) \xi_0} y \phi (y,\xi_0) dy \\
& + & \int K(y) e^{i (x-y) \xi_0} \phi (x-y,\xi_0) dy \\
& = & \left( \frac{1}{\sqrt{\xi_0^2 + 2 \lambda^2}} - \frac{1}{|\xi_0|} \right) \int K(y) e^{i y \xi_0} y \phi (y,\xi_0) dy \\
& + & \int K(y) e^{i (x-y) \xi_0} \phi_{\xi_0} (x-y,\xi_0) dy,
\end{eqnarray*}
where we have used $i \xi_0 e^{i y \xi_0} = \partial_y e^{i y \xi_0}$ and integrated by parts.  As a result,
\begin{eqnarray*}
\| \partial^\beta_{\xi_0} f_0 \|_{L^{\frac43}} \leq |\xi_0|^{2-|\beta|} C_\beta,
\end{eqnarray*}
for any multi-index $\beta$, $|\beta|=0,1,2,\dots$.  Combining the above results, we have
\begin{eqnarray*}
|\partial^\alpha_\xi \partial^\beta_x f_0 (x,\xi) | \leq C_{\alpha, \beta} |\xi|^{2 - |\alpha|},
\end{eqnarray*}
or $f_0 \in S^2$, which gives \eqref{eqn:symb}.

For the spatial regularity result, we once again use that the distorted Fourier basis satisfies the equation
\begin{eqnarray*}
g_{\xi_0} = K*(Fe^{ix\cdot \xi_0}) + K*(\tilde{V} g_{\xi_0}).
\end{eqnarray*}
We have existence for $g_{\xi_0}$ in $L^4$, but we can take advantage of the structure of $K*P$ in order to show improved regularity.  Then,
\begin{eqnarray*}
\nabla g_{\xi_0} = (\nabla K)*(Fe^{i x \cdot \xi_0}) + (\nabla K)*(\tilde{V} g_{\xi_0}).
\end{eqnarray*}
Hence, we must explore the nature of $(\nabla K)*(\tilde{V})$.  
Upon differentiating, we see
\begin{eqnarray*}
(\nabla K) (x-y) = O( |x-y|^{-1}),
\end{eqnarray*}
which means by a similar approach to Section \ref{lin:dfb}, we get
\begin{eqnarray*}
\| \nabla g_{\xi_0} \|_{L^4} \leq C( \| F \|_{L^\frac{12}{11}} + \| V \|_{L^{\frac{3}{2}}} \| g_{\xi_0} \|_{L^4}).
\end{eqnarray*}
To see this, we first use the Hardy-Littlewood-Sobolev inequality (see \cite{Ste}) with $\gamma = 1$ so
\begin{eqnarray*}
\frac{1}{p} = \frac{2}{3} + \frac{1}{4} = \frac{11}{12},
\end{eqnarray*}
then H\"olders inequality such that
\begin{eqnarray*}
\| V g \|_{L^{\frac{12}{11}}} \leq \| V \|_{L^{\frac{3}{2}}} \| g \|_{L^4}.
\end{eqnarray*}
Then, we can iterate this for all derivatives and using Sobolev embeddings, get continuity of all derivatives and hence smoothness.

To prove existence for $\partial_{\xi_0} g_{\xi_0}$ in Sobolev spaces, we must show that $\partial_{\xi_0} g_{\xi_0}$ is defined and bounded in some space of functions.  In this direction, we look at
\begin{eqnarray*}
[(-\Delta + 2 \lambda^2 + \xi_0^2)(-\Delta-\xi_0^2)] g_{\xi_0} & = & F_{\xi_0} e^{i x \xi_0} + \tilde{V} g_{\xi_0}
\end{eqnarray*}
and
\begin{eqnarray*}
[(-\Delta + 2 \lambda^2 + (\xi_0+h_j)^2)(-\Delta-(\xi_0+h_j)^2)] g_{\xi_0+h_j} & = & F_{\xi_0+h_j} e^{i x (\xi_0+h_j)} + \tilde{V} g_{\xi_0+h_j},
\end{eqnarray*}
where $h_j = h e_j$ and $e_j$ is the unit vector in the $j$-th coordinate.
Hence, if we define
\begin{eqnarray*}
v_{h} = g_{\xi_0+h_j} - g_{\xi_0},
\end{eqnarray*}
then we must solve
\begin{eqnarray*}
L_{\xi_0} (v_h) & = & (F_{\xi_0+h_j} e^{i x \cdot (\xi_0+h_j)} - F_{\xi_0} e^{i x \xi_0}) + O(h) u_{\xi_0} + \tilde{V} (v_h) \\
& = & O(h) (\tilde{F}_{\xi_0} + F_{\xi_0} + K* \tilde{V} u_{\xi_0}) + \tilde{V} (v_h).
\end{eqnarray*}
We can write this as 
\begin{eqnarray*}
L_{\xi_0} [v_h - O(h) K*(K*( \tilde{V} g_{\xi_0}))]  = O(h) (G) + \tilde{V} [v_h - O(h) K*(K*( \tilde{V} g_{\xi_0}))],
\end{eqnarray*}
where we have
\begin{eqnarray*}
G =\tilde{F}_{\xi_0} + F_{\xi_0} - \tilde{V} K*(K*(\tilde{V} g_{\xi_0})).
\end{eqnarray*}
To see that $G \in L^4$, we need only see that
\begin{eqnarray*}
\|\tilde{V} K*(K*(\tilde{V} g_{\xi_0}))\|_{L^4} < \infty
\end{eqnarray*}
since the other terms are dealt with above in the spatial regularity analysis.  However,
we have
\begin{eqnarray*}
K*(K*\cdot):L^1 \to L^\infty,
\end{eqnarray*}
following analysis similar to the complex interpolation argument.  Also, by moving all of the derivatives onto $Pu$, we see this is smooth.  All we lack is nice decay, hence
\begin{eqnarray*}
\|\tilde{V} K*(K*(\tilde{V} g_{\xi_0}))\|_{L^4} < \| K*(K*(\tilde{V} g_{\xi_0})) \|_{L^\infty} \| V \|_{L^4},
\end{eqnarray*}
for $V \in \mathcal{S}$ as given in the description of $P$.
From the Fredholm Theory, we know 
\begin{eqnarray*}
\| \frac{v_h}{h} - O(1) K*(K*(\tilde{V} g_{\xi_0})) \|_{L^4} \leq C,
\end{eqnarray*}
for $C = C(\xi_0)$.  However, given $w \in C^\infty_0 \cup L^4$ a sufficiently decaying, smooth function, we have
\begin{eqnarray*}
\| w \frac{v_h}{h} \|_{L^4} & \leq & C (1 + \| w K*(K*(\tilde{V} g_{\xi_0})) \|_{L^4} ) \\
& \leq & C
\end{eqnarray*}
from Section \ref{lin:dfb}, where $C$ is independent of $h$.  In this case, we have 
\begin{eqnarray*}
K*(K*(\tilde{V} g_{\xi_0})) \in L^\infty
\end{eqnarray*}
using H\"older's inequality, so we can take $w = \langle x \rangle^{-1}$.
Thus, we can take the limit as $h \to 0$ to see that derivatives in $\xi_0$ are bounded in weighted $L^4$ spaces.  Iterating this process involves taking stonger weight functions at each step of the iteration.  As a result, since $\tilde{V}$ has exponentially decaying terms in $x$ and $\tilde{V} g_{\xi_0}$ is well-defined in $L^4$ from the spatial regularity, we have the desired regularity in $\xi_0$.  

Now that we have differentiability with respect to $\xi_0$,
\begin{eqnarray*}
\partial_{(\xi_0)_j} \left( [(-\Delta + 2 \lambda^2 + \xi_0^2)(-\Delta-\xi_0^2)] g_{\xi_0}  = Fe^{i x \xi_0} + \tilde{V} g_{\xi_0} \right) 
\end{eqnarray*}
which implies
\begin{eqnarray*}
L_{\xi_0} \partial_{(\xi_0)_j} g_{\xi_0}& = & \partial_{(\xi_0)_j} (Fe^{i x \xi_0}) + P  \partial_{(\xi_0)_j} g_{\xi_0} \\
& - & 2 (\xi_0)_j (-\Delta-\xi_0^2) g_{\xi_0} - (\xi_0) (-\Delta + 2 \lambda^2 + \xi_0^2) g_{\xi_0}.
\end{eqnarray*}

For higher derivatives in $\xi_0$, we iterate this procedure.

\end{proof}

\begin{rem}
Note that the above analysis can also be done in the case where instead of $L^4$ we use $L^2(\langle x \rangle^{-s})$ as in \cite{Ag}.  To see this, note that
\begin{eqnarray*}
\| \phi \|_{L^1} \lesssim \| \phi \|_{L^2 (\langle x \rangle^{s})},
\end{eqnarray*}
where $s > d$, and
\begin{eqnarray*}
\| \phi \|_{L^2 (\langle x \rangle^{-s})} \lesssim \| \phi \|_{L^\infty},
\end{eqnarray*}
where $s >d$.  Then, we can go to the Sobolev norms to apply Hardy-Littlewood-Sobolev and use H\"older's inequality in weighted spaces and the boundedness of $V_1$ and $V_2$ in weighted $L^2$ spaces to complete the argument.
\end{rem}

\begin{rem}
\label{rem:dfbasy}
As $x \to \infty$, note that since $V_1$, $V_2 \in \mathcal{S}$, using Equation \eqref{eqn:dfbeq}, we have 
\begin{eqnarray*}
u_{\xi_0} \to \frac{\pi^2}{|\xi_0|^2+\lambda^2} \left[\frac{e^{\pm i |x| |\xi_0|} - e^{-|x|\sqrt{|\xi_0|^2+2 \lambda^2}}}{|x|}\right],
\end{eqnarray*}
which explains the choice of spaces $L^{2,s}$ for $x > \frac{1}{2}$ in \cite{Ag}.
\end{rem}

\section{Representation of the solution}
\label{lin:rep}

We present here a slightly different approach to the distorted Fourier transform, though the motivation comes from \cite{Ho2}.  

\begin{thm}
\label{thm:rep}
For $V \in \mathcal{S}$, there exists a distorted Fourier basis $\tilde{\phi}_\xi$ and correspondingly a distorted Fourier transform $\mathcal{G}$ for the nonselfadjoint operator $\mathcal{H}$, where
\begin{eqnarray*}
\mathcal{G}_\pm f = \int \tilde{\phi}^\pm_{\xi} (x) f (x) dx.
\end{eqnarray*}
Similarly, there exists an inverse Fourier basis $\tilde{\phi}_\xi^{-1} (x) $ and correspondingly an inverse Fourier transform $\mathcal{G}^{-1}$ for the nonselfadjoint operator $\mathcal{H}$, where
\begin{eqnarray*}
\mathcal{G}_\pm^{-1} f = \int \{ \tilde{\phi}^\pm_\xi \}^{-1} (x) f (\xi) d\xi.
\end{eqnarray*}
It follows that
\begin{eqnarray*}
\| \mathcal{G}_\pm \|_{L^2 \to L^2} & \lesssim & 1, \\
\| \mathcal{G}_\pm^{-1} \|_{L^2 \to L^2} & \lesssim & 1.
\end{eqnarray*}
These operators are not unitary, however 
\begin{eqnarray*}
\| \mathcal{G}_\pm^{-1} \mathcal{G} \|_{L^2 \to L^2} \lesssim 1
\end{eqnarray*}
and 
\begin{eqnarray*}
\mathcal{G}_\pm^{-1} \mathcal{G}_\pm \phi = P_c \phi.
\end{eqnarray*}
\end{thm}  

Before we prove the theorem, look at the operator 
\begin{eqnarray*}
\mathcal{H}^2 = \left[ \begin{array}{cc}
L_{-} L_{+} & 0 \\
0 & L_{+} L_{-}
\end{array} \right],
\end{eqnarray*}
for which we have the following self-adjoint realization
\begin{eqnarray*}
\tilde{\mathcal{H}} = \left[ \begin{array}{cc}
L_{-}^{\frac12} L_{+} L_{-}^{\frac12} & 0 \\
0 & L_{-}^{\frac12} L_{+} L_{-}^{\frac12}
\end{array} \right].
\end{eqnarray*}

Since
\begin{eqnarray*}
L_{-}^{\frac12} L_{+} L_{-}^{\frac12} & = & (-\Delta + \lambda^2 - V_1)^{\frac12} (-\Delta + \lambda^2 - V_1 - V_2) (-\Delta + \lambda^2 - V_1)^{\frac12} \\
& = & (-\Delta + \lambda^2 - V_1)^{2} - (-\Delta + \lambda^2 - V_1)^{\frac12} V_2 (-\Delta + \lambda^2 - V_1)^{\frac12} \\
& = & L_{-}^2 - L_{-}^{\frac{1}{2}} V_2 L_{-}^{\frac{1}{2}}.
\end{eqnarray*}
This is a fourth order constant coefficient operator with a lower order perturbation.  However, the perturbation is no longer a differential operator.  Ideally, by a similar analysis to that in \cite{Ag}, there exists a distorted Fourier basis, say $\tilde{u}_\xi$ such that
\begin{eqnarray*}
L_{-}^{\frac12} L_{+} L_{-}^{\frac12} \tilde{u}_\xi = (\lambda^2 + \xi^2)^2 \tilde{u}_\xi.
\end{eqnarray*}
To prove this, we need to show $L_{-}^{\frac12}$ is a pseudodifferential operator of strong enough class, which we explore in the sequel.  

From Theorem \ref{thm:dfbu}, we have $u_\xi = e^{ix \xi} + f_\xi (x)$, $v_\xi = e^{ix \xi} + g_\xi (x)$ such that
\begin{eqnarray*}
\mathcal{H}^2 \left[ \begin{array}{c}
u_\xi \\
v_\xi
\end{array} \right] = (\lambda^2 + \xi^2)^2 \left[ \begin{array}{c}
u_\xi \\
v_\xi
\end{array} \right],
\end{eqnarray*}
where $f_\xi (x)$, $g_\xi (x) \in L^4_x$, smooth in $x$ and $\xi$, and 
\begin{eqnarray*}
f_\xi, g_\xi \sim \frac{\pi^2}{|\xi_0|^2+\lambda^2} \left[\frac{e^{\pm i |x| |\xi_0|} - e^{-|x|\sqrt{|\xi_0|^2+2 \lambda^2}}}{|x|}\right]
\end{eqnarray*}
as $x \to \infty$.

Formally, we would like to say
\begin{eqnarray*}
\left[ \begin{array}{cc}
L_{-}^{\frac12} L_{+} L_{-}^{\frac12} & 0 \\
0 & L_{-}^{\frac12} L_{+} L_{-}^{\frac12}
\end{array} \right] \left[ \begin{array}{c}
L_{-}^{-\frac12} u_\xi \\
L_{-}^{\frac12} v_\xi
\end{array} \right] =  (\lambda^2 + \xi^2)^2 \left[ \begin{array}{c}
L_{-}^{-\frac12} u_\xi \\
L_{-}^{\frac12} v_\xi
\end{array} \right],
\end{eqnarray*}
however as $u_\xi$, $v_\xi \notin L^2$, we must investige further.

Before we begin, let us analyze the connection between $u_\xi$ and $v_\xi$.  For instance,
\begin{eqnarray*}
L_{+} (L_{-} L_{+} u_\xi) & = & L_{+} L_{-} (L_{+} u_\xi) \\
& = & L_{+} (\lambda^2 + \xi^2)^2 u_\xi, \\
L_{-} (L_{+} L_{-} v_\xi) & = & L_{-} L_{+} (L_{-} v_\xi) \\
& = & L_{-} (\lambda^2 + \xi^2)^2 v_\xi.
\end{eqnarray*}
Hence
\begin{eqnarray*}
L_{+} u_\xi = C v_\xi
\end{eqnarray*}
and 
\begin{eqnarray*}
L_{-} v_\xi = C u_\xi.
\end{eqnarray*}

In particular, we are interested in
\begin{eqnarray*}
L_{-} v_\xi & = & (-\Delta + \lambda^2 - V_1) (e^{i x \xi} + g_\xi) \\
& = & (\xi^2 + \lambda^2) e^{i x \xi} + L_{-} g_\xi - V_1 e^{ix \xi}, \\
C u_\xi & = & C (e^{ix \xi} + f_\xi).
\end{eqnarray*}
Then, $C = (\lambda^2 + \xi^2)$, so 
\begin{eqnarray*}
L_{-} v_\xi = (\lambda^2 + \xi^2) u_\xi, 
\end{eqnarray*}
\begin{eqnarray*}
L_{-}^{-1} u_\xi = (\lambda^2 + \xi^2)^{-1} v_\xi, 
\end{eqnarray*}
and 
\begin{eqnarray*}
f_\xi = \frac{1}{\lambda^2 + \xi^2} (L_{-} g_\xi - V_1 e^{ix \xi}).
\end{eqnarray*}
A similar calculation holds for $L_{+} u_\xi = C v_\xi$.

Note also that if we look at the vector
\begin{eqnarray*}
\vec{\phi}_\xi = \left[ \begin{array}{c}
i u_\xi \\
v_\xi
\end{array} \right],
\end{eqnarray*}
then we have
\begin{eqnarray*}
\mathcal{H} \vec{\phi}_\xi = (\lambda^2 + \xi^2) \vec{\phi}_\xi.
\end{eqnarray*}

To be more precise, we say that the operator $L_{-}^{\frac12} L_{+} L_{-}^{\frac12}$ has a distorted Fourier basis given by $\tilde{u}_\xi$, then find an expression for the distorted Fourier transform of $\mathcal{H} P_c$.  This distorted Fourier transform will be defined via a distorted Fourier basis that will give the relationship between $\tilde{u}_\xi$, $u_\xi$ and $v_\xi$.  The existence of $\tilde{u}_\xi$ must be proved since there is a lower order PDO perturbation instead of a differential operator.  See \cite{Ho2}.

In order to prove $L_{-}^{\pm \frac12}$ is a PDO, we must use a result similar to one from \cite{Ho4}, Chapter 29.  To this end, we refer to the following theorem given in \cite{Ho4}: 

\begin{thm}
Let $X$ be a compact manifold, $\Psi$ a space of pseudo-differential operators and $\Omega^{\frac{1}{2}}$ be the space of half-densities on $X$.  Let $P \in \Psi^m_{phg} (X; \Omega^{\frac{1}{2}}, \Omega^{\frac{1}{2}})$ be a positive, elliptic, symmetric operator.  Then, $P$ defines a positive, self-adjoint operator $\mathcal{P}$ in $L^2 (X,\Omega^{\frac{1}{2}}$.  If $m > 0$ and $a \in \R$, then $\mathcal{P}^a$ is also defined by a pseudodifferential operator in $\Psi^{am}_{phg} (X; \Omega^{\frac{1}{2}}, \Omega^{\frac{1}{2}})$, with principal and subprincipal symbols $p^a$ and $a p^{a-1} p^s$ if $p$ and $p^s$ are those for $P$.  
\end{thm}

We seek to prove a slightly different version here:

\begin{thm}
Let $P$ be a positive, symmetric, self-adjoint operator in $\Psi^{m,(2)}_{\rho, \delta} (\R^d)$.  Then, $P$ defines a postive, self-adjoint operator $\mathcal{P}$ in $L^2 (\R^d,\R^d)$.  If $m>0$ and $a \in \R$, then $\mathcal{P}^a$ is also defined by a pseudodifferential operator in $\Psi^{am,(2)}_{\rho, \delta} (\R^d,\R^d)$, with principal and subprincipal symbols $p^a$ and $a p^{a-1} p^s$ if $p$ and $p^s$ are those for $P$.
\end{thm}

Note that since $R \in \mathcal{S}$, $F(R) \in \mathcal{S}$ by the properties of the nonlinearity.  Hence, we have the following:

\begin{lemma}
\label{thm1:lemsr}
The perturbation $V_1$ is short-range.
\end{lemma}

We need to prove that given the operator, 
\begin{eqnarray*}
L_{-} = -\Delta + \lambda^2 - V_1 \in S^{2},
\end{eqnarray*}
the new operator $L_{-}^a$ is a pseudodifferential operator for $a \in \reals$.

\begin{lemma}
For an operator $P$, the resolvent $R(z) = (P - z)^{-1}$ exists and is analytic for all $z$ except the eigenvalues of $P$.  Also, $\| R(z) \|_{L^2 \to L^2}$ is bounded by the inverse of the distance from $z$ to the nearest eigenvalue.
\end{lemma}

\begin{proof}
This follows from basic facts from spectral theory as discussed in \cite{HS}.
\end{proof}

\begin{thm}
\label{thm:rep1}
The operator $L_{-}^a$ is pseudodifferential operator in the class $S^{2a}$ for $a \in \reals$.
\end{thm}

Before we prove the theorem, let us prove the following lemma from \cite{Ho3}.

\begin{lemma}
Let $a \in S^m$.  If 
\begin{eqnarray}
\label{eqn:par}
|a(x,\xi)| > c |\xi|^m
\end{eqnarray}
for $|\xi| > C$, then there exists $b \in S^{-m}$ such that
\begin{eqnarray*}
(i) \ a(x,\xi) b(x,\xi) - 1 \in S^{-1}, \\
(ii) \ a(x,D) b(x,D) - I \in Op S^{-\infty},
\end{eqnarray*}
and
\begin{eqnarray*}
(iii) \ b(x,D) a(x,D) - I \in Op S^{-\infty}.
\end{eqnarray*}
\end{lemma}

\begin{proof}[Proof of Lemma]

First, let us prove that \eqref{eqn:par} implies $(i)$.  We can reduce this to the case where $m=0$ by looking at $a(x,\xi) (1+|\xi|^2)^{-m/2}$ and $b(x,\xi) (1+|\xi|^2)^{m/2}$. 

\begin{claim}
If $a_1$, $a_2 \in S^0$ and $F \in C^\infty (\mathbb{C}^2)$, then $F(a_1,a_2) \in S^0$.  
\end{claim}

\begin{proof}
Since the $\Re a_j$, $\Im a_j \in S^0$ for $j = 1,2$, we may assume that $a_j$ is real and $F \in C^\infty (\reals^2)$.  Then,
\begin{eqnarray*}
\frac{\partial F(a)}{\partial x_j} & = & \sum_k \frac{\partial F}{\partial a_k} {\partial_{x_j} a_k}, \\ 
\frac{\partial F(a)}{\partial \xi_j} & = & \sum_k \frac{\partial F}{\partial a_k} {\partial_{\xi_j} a_k},
\end{eqnarray*}
where $\partial_{x_j} a_k \in S^0$, $\partial_{\xi_j} a_k \in S^{-1}$.  Hence, it is clear the derivatives of $F(a)$ decay as necessary for $F(a)$ to be in $S^0$.
\end{proof}

Hence, for $m = 0$, choose $F \in C^\infty$ so that $F(z) = \frac{1}{z}$ for $|z| > c$.  Set $b = F(a) \in S^0$ so $a(x,\xi) b(x,\xi) = 1$ for $|\xi| > C$.  This proves $(i)$.  

Using $(i)$, we have that 
\begin{eqnarray*}
a(x,D) b(x,D) = I - r(x,D), \ r \in S^{-1}.
\end{eqnarray*}
Set 
\begin{eqnarray*}
b(x,D) r(x,D)^k = b_k (x,D), \ b_k \in S^{-m-k},
\end{eqnarray*}
so we can iterate out the error.  Let $b'$ be the asymptotic sum of the $b_k$'s, so
\begin{eqnarray*}
a(x,D) b' (x,D) - I = a(x,D) (b' (x,D) - \sum_{j<k} b_j (x,D)) - r(x,D)^k \in Op S^{-k},
\end{eqnarray*}
for every $k$.  Then, we have $(ii)$ replacing $b$ with $b'$.  Similarly, we can find a $b''$ which satisfies $(iii)$.  Note also that
\begin{eqnarray*}
b' - b'' = b' (I - ab') + (b'' a - I) b',
\end{eqnarray*}
hence $b'$ and $b''$ are equivalent modulo $S^{-\infty}$.  
\end{proof}

\begin{proof}[Proof of Theorem \ref{thm:rep1}]

Since $L_{-}$ is self-adjoint, we have that $R(z)$ is defined and analytic for all $z$ except at the eigenvalues of $L_{-}$.  The $L^2$ norm of the resolvent can be estimated by the inverse of the distance to the set of eigenvalues.  Now, since $a < 0$, we have by the spectral theorem
\begin{eqnarray*}
\tilde{L}^a u = -(2 \pi i)^{-1} \int_{- i \infty}^{i \infty} z^a R(z) u dz,
\end{eqnarray*}
where the contour is slightly deformed near the origin to avoid $z=0$ and $z^a$ is analytic in the right half plane and equal to $1$ when $z = 1$.  Since $L_{-}^{a+1} = L_{-}^1 L_{-}^a$, the distribution kernel of $L_{-}^a$ is an entire analytic function of $a$.  

To understand the behavior of the singularities, we construct a parametrix.  Namely, since $|L_{-} (x,\xi)| > c |\xi|^2$ for $|\xi| > C$, we have the existence of an inverse modulo $S^{-1}$.  Then, we can iterate that error, to find an inverse modulo $S^{-\infty}$.  

In particular, we have $B_z$ such that
\begin{eqnarray*}
(P-z) B_z = I - Q_z,
\end{eqnarray*}
where $b_z = F(P(x,\xi)-z)$, $F(z) \sim 1/|z|$ for $z$ large and $Q_z \in Op (S^{-1})$.  Then, there is an $E_z$ given by the asymptotic sum  
\begin{eqnarray*}
\sum_{j=0}^\infty B_z (x,D) (Q_z (x,D))^j,
\end{eqnarray*}
such that 
\begin{eqnarray*}
(P - z) E_z = I - W_z,
\end{eqnarray*}
where $W_z \in Op (S^{-\infty})$.  So, we have
\begin{eqnarray*}
R(z) = E_z + R(z) W_z.
\end{eqnarray*}
Then, for $a < 0$, we have
\begin{eqnarray*}
\tilde{L}^a = -(2 \pi i)^{-1} \int_{-i \infty}^{i \infty} z^a E_z dz + T(a) u.
\end{eqnarray*}  
Here, $T(a)$ should be analytic in $a$ for $a < 1$.  In particular, this remainder will be a well-behaved pseudo-differential operator using Beals' Theorem as discussed in \cite{Beals}.  From the composition of pseudodifferential operators, we have that 
\begin{eqnarray*}
Q_z = \sum_{\alpha > 0} \partial_\xi^\alpha L_{-} (x,\xi) \partial_x^\alpha F( L_{-} - z ) / \alpha !.
\end{eqnarray*}
Hence, the terms of $E_z$ outside of compact set in phase space look like
\begin{eqnarray*}
(P-z)^{-k-1} q
\end{eqnarray*}
where $q \in S^{mk-\kappa}$ for some $\kappa \geq 0$.  

Hence, there is a pseudodifferential operator representation of $L_{-}^{-\frac{1}{2}}$ and thus $L_{-}^{\frac{1}{2}}$ by multiplication by the operator.  If $p$ is the principal symbol of $L_{-}$, the principal symbol of $L_{-}^a$ will be $F(p)$ where $F(z) = z^a$ for $|z|>C$.

\end{proof}

\begin{lemma}
The pseudodifferential operator $L_{-} V_1 + V_1 (-\Delta + \lambda^2) + L^{\frac{1}{2}}_{-} V_2 L^{\frac{1}{2}}_{-}$ is a short range perturbation.
\end{lemma}

\begin{proof}
This proof should be similar to that in Lemma \ref{thm1:lemsr}.  The argument for the differential operator $L_{-} V_1 + V_1 (-\Delta + \lambda^2)$ follows precisely as above.  Hence, we focus only on the compactness and iteration arguments for the pseudodifferential operator, $L^{\frac{1}{2}}_{-} V_2 L^{\frac{1}{2}}_{-}$.  In what follows, let $Tu = L_{-}^{\frac12} V L_{-}^{\frac12} K * (u)$.  In particular, we need to prove:
\begin{eqnarray}
\label{eqn:pdo1}
L_{-}^{\frac12} V L_{-}^{\frac12} e^{i x \cdot \xi_0} & = & e^{i x \cdot \xi_0} V_{\xi_0}, \ \text{for} \ V_{\xi_0} \in \mathcal{S}, \ \| V_{\xi_0} \|_L^\infty \sim |\xi_0|^2, \\
\label{eqn:pdo2}
\| K * T^n (e^{i x \cdot \xi_0} V_{\xi_0}) \|_{L^4} & = & O(|\xi_0|^{-\frac{n+1}{2}}), \\
\label{eqn:pdo3}
K* ( L_{-}^{\frac12} V L_{-}^{\frac12} \cdot) & : & W^{2,4} \hookrightarrow W^{2,4}.
\end{eqnarray}

For \eqref{eqn:pdo1}, we have in the sense of distributions that 
\begin{eqnarray*}
\mathcal{F} e^{ i x \xi_0} = \delta_{\xi_0} (\xi).
\end{eqnarray*}
Hence, since $V \in \mathcal{S}$,
\begin{eqnarray*}
L_{-}^{\frac12} V L_{-}^{\frac12} & = & \int P(x,\xi) e^{ i (x-x_1) \xi} V(x_1) \int P(x_1,\xi_1) e^{i x_1 \xi_1} \delta_{\xi_0} (\xi_1) d\xi_1 dx_1 d\xi \\
& = & \int P(x,\xi) e^{ i (x-x_1) \xi} V(x_1) \int P(x_1,\xi_0) e^{i x_1 \xi_0} dx_1 d\xi \\
& = & e^{i x \xi_0} \tilde{V} (x,\xi_0) + l.o.t.,
\end{eqnarray*}
where $\tilde{V} \in \mathcal{S}(x)$ and $| \tilde{V} | \lesssim \xi_0^2$ precisely as in Section \ref{lin:dfb}.  This comes in particular from realizing that the principal symbol of $L_{-}^{\frac12} V L_{-}^{\frac12}$ is
\begin{eqnarray*}
(\xi^2+\lambda^2-V_1(x)) V_2 (x).
\end{eqnarray*}

The results \eqref{eqn:pdo2} and \eqref{eqn:pdo3} follow from the following theorem proved in \cite{Ste}, Chapter VI.  

\begin{thm}[Stein]
Suppose $T_a$  is a pseudo-differential operator whose symbol $a$ belongs to $S^{m}$.  If $m$ is an integer and $k \geq m$, then $T_a$ is a bounded mapping from $W^{k,p} \to W^{k-m,p}$, whenever $1 < p < \infty$.  
\end{thm}

Since $L_{-}^{\frac12} \in S^1$ and $V \in S^0$, we have $L_{-}^{\frac12} V L_{-}^{\frac12} \in S^2$, hence
\begin{eqnarray*}
L_{-}^{\frac12} V L_{-}^{\frac12}: W^{2,4} \to L^4.
\end{eqnarray*}  
As $V \in \mathcal{S}$, we in fact have more than this.  Define the symbol class
\begin{eqnarray*}
S^m_r = \{ p | p \in S^m, \ |x^\alpha \partial_x^\beta \partial_\xi^\gamma p(x,\xi)| \leq C_{\alpha, \beta} |\xi|^{m-\gamma} \}.
\end{eqnarray*}
In other words, we have the standard symbol class $S^m$, where the symbol has rapid decay in $x$.  Here, $V \in S^0_r$.  Note that due to the properties of Schwarz class functions, we have for $p \in S^{m_1}$ and $q \in S^{m_2}$,
\begin{eqnarray*}
p q, qp \in S^{m_1 + m_2}_r
\end{eqnarray*}
and
\begin{eqnarray*}
q u : W^{m_2,p} \to L^q,
\end{eqnarray*}
where $1 < p,q < \infty$.

For \eqref{eqn:pdo2}, from the analysis in Theorem \ref{thm:dfbu} we have
\begin{eqnarray*}
(K * \cdot) : L^{\frac{4}{3}} \to W^{2,4}.
\end{eqnarray*}
We have from \eqref{eqn:pdo1}
\begin{eqnarray*}
\| \int K(x-y) e^{iy \cdot \xi_0} V_{\xi_0} (y) dy \|_{L^4} \lesssim |\xi_0|^{-\frac{1}{2}}.
\end{eqnarray*}
Then,
\begin{eqnarray*}
\| \int K(x-y) L_{-}^{\frac12} V(y) L_{-}^{\frac12} \int K(y-z) e^{iz \cdot \xi_0} V_{\xi_0} (z) dz dy \|_{L^4} \\
= |\xi_0|^{-\frac{1}{2}} \| L_{-}^{\frac12} V(y) L_{-}^{\frac12} \int K(y-z) e^{iz \cdot \xi_0} V_{\xi_0} (z) dz dy \|_{L^{\frac{4}{3}}} \\
\lesssim |\xi_0|^{-\frac{1}{2}} \| \int K(y-z) e^{iz \cdot \xi_0} V_{\xi_0} (z) dz dy \|_{W^{2,4}} \\
\lesssim  |\xi_0|^{-1} \| e^{iz \cdot \xi_0} V_{\xi_0} (z) \|_{L^{\frac{4}{3}}},
\end{eqnarray*}
using the fact that $V \in S^0_r$ and the mapping properties of $K$ described in Theorem \ref{thm:dfbu}.  Iterating this procedure, we get the result.

For \eqref{eqn:pdo3}, if $u \in W^{2,4}$, 
\begin{eqnarray*}
\| L_{-}^{\frac12} V L_{-}^{\frac12} u \|_{L^{\frac{4}{3}}} \lesssim \| u \|_{W^{2,4}}.
\end{eqnarray*}
By decay properties of $V$, we have
\begin{eqnarray*}
\| x L_{-}^{\frac12} V L_{-}^{\frac12} u \|_{L^{\frac{4}{3}_x}} \lesssim \| u \|_{W^{2,4}} + \| u \|_{W^{1,4}} \lesssim \| u \|_{W^{2,4}}.
\end{eqnarray*}
The inherent integration by parts is justified as $V \in \mathcal{S}$.
Hence,  by iterating this procedure and using properties of convolutions,
\begin{eqnarray*}
[K * (L_{-}^{\frac12} V L_{-}^{\frac12} \cdot)] : W^{2,4} \to W^{2,4} (\langle \cdot \rangle^{N}),
\end{eqnarray*}
for any $N \in \mathbb{N}$.  However, $W^{2,4} (\langle \cdot \rangle^{N})$ is compactly embedded in $W^{2,4}$, so \eqref{eqn:pdo3} holds.

\end{proof}

\begin{lemma}
There exists a distorted Fourier basis, $\tilde{u}_\xi$, for $L^{\frac{1}{2}}_{-} L_{+} L^{\frac{1}{2}}_{-}$ with the aforementioned smoothness properties.
\end{lemma}

\begin{proof}
Apply the techniques from the proof of Theorem \ref{thm:dfbu}, applying \eqref{eqn:pdo1}, \eqref{eqn:pdo2}, and \eqref{eqn:pdo3} when necessary.  Once the compactness is established, the standard self-adjoint techniques are available to give
\begin{eqnarray*}
\| P_c \phi \|_{L^2}^2 & = & (2 \pi)^{-d} \int | \mathcal{F}_\pm \phi |^2 dx, \\
\mathcal{F}_\pm^{-1} P_0 \mathcal{F}_\pm \phi & = & P_c \phi,
\end{eqnarray*}
where $\mathcal{F}_\pm$ is the distorted Fourier transform associated to $\tilde{u}^\pm_{\xi_0}$ and $P_0 (\xi_0) = (\xi_0^2 + \lambda)^2$ is the symbol for the leading order constant coefficient operator.

\end{proof}

Since
\begin{eqnarray*}
\left[ \begin{array}{cc}
L_{-}^{\frac12} L_{+} L_{-}^{\frac12} & 0 \\
0 & L_{-}^{\frac12} L_{+} L_{-}^{\frac12}
\end{array} \right] = \left[ \begin{array}{cc}
L_{-}^{-\frac12} & 0 \\
0 &  L_{-}^{\frac12}
\end{array} \right] \left[ \begin{array}{cc}
L_{-} L_{+} & 0 \\
0 & L_{+} L_{-}
\end{array} \right] \left[ \begin{array}{cc}
L_{-}^{\frac12} & 0 \\
0 &  L_{-}^{-\frac12}
\end{array} \right],
\end{eqnarray*}
we have
\begin{eqnarray*}
\left[ \begin{array}{cc}
L_{-}^{-\frac12} & 0 \\
0 &  L_{-}^{\frac12}
\end{array} \right] \mathcal{H}^2 \left[ \begin{array}{cc}
L_{-}^{\frac12} & 0 \\
0 &  L_{-}^{-\frac12}
\end{array} \right] P_{c} f = \tilde{\mathcal{F}}^*_{\pm} |(\xi^2+\lambda^2)^2| \tilde{\mathcal{F}}_{\pm} f,
\end{eqnarray*}
where $\tilde{\mathcal{F}}_{\pm}$ is the distorted Fourier transform with respect to $\tilde{u}_\xi$.  Setting
\begin{eqnarray*}
f = \left[ \begin{array}{cc}
L_{-}^{-\frac12} & 0 \\
0 &  L_{-}^{\frac12}
\end{array} \right] \tilde{f},
\end{eqnarray*}
we see
\begin{eqnarray*}
\left[ \begin{array}{cc}
L_{-}^{-\frac12} & 0 \\
0 &  L_{-}^{\frac12}
\end{array} \right] \mathcal{H}^2 \left[ \begin{array}{cc}
L_{-}^{\frac12} & 0 \\
0 &  L_{-}^{-\frac12}
\end{array} \right] P_{c} f = \tilde{\mathcal{F}}^*_{\pm} |(\xi^2+\lambda^2)^2| \tilde{\mathcal{F}}_{\pm} \left[ \begin{array}{cc}
L_{-}^{-\frac12} & 0 \\
0 &  L_{-}^{\frac12}
\end{array} \right] \tilde{f}.
\end{eqnarray*}
Hence,
\begin{eqnarray*}
\mathcal{H}^2 (P_{c} \tilde{f}) = \left[ \begin{array}{cc}
L_{-}^{\frac12} & 0 \\
0 &  L_{-}^{-\frac12}
\end{array} \right] \tilde{\mathcal{F}}^*_{\pm} |(\xi^2+\lambda^2)^2| (\tilde{\mathcal{F}}_{\pm} \left[ \begin{array}{cc}
L_{-}^{-\frac12} & 0 \\
0 &  L_{-}^{\frac12}
\end{array} \right]) \tilde{f},
\end{eqnarray*}
or
\begin{eqnarray*}
\mathcal{H}^2 (P_{c} \tilde{f})(x) = \left[ \begin{array}{c}
\int (L_{-}^{\frac12} \bar{\tilde{u}}_\xi) (x) |(\xi^2+\lambda^2)^2| \int \tilde{u}_\xi (y) (L_{-}^{-\frac12} \tilde{f}_1)(y) dy d\xi \\
\int (L_{-}^{-\frac12} \bar{\tilde{u}}_\xi) (x) |(\xi^2+\lambda^2)^2| \int \tilde{u}_\xi (y) (L_{-}^{\frac12} \tilde{f}_2) (y) dy d\xi.
\end{array} \right]
\end{eqnarray*}
The inverse operations in these arguments are justified by the fact that 
\begin{eqnarray*}
L^2 = \text{Ker}({\mathcal{H}}) \oplus \text{Ker}({\mathcal{H}^*})^\perp.
\end{eqnarray*}

We desire an oscillatory integral formulation for $\mathcal{H} P_c$.  The continuous spectrum is spanned by the values $\pm (\lambda^2 + \xi^2)$ for all $|\xi| \geq 0$.  Hence, we seek a diagonalization of the form
\begin{eqnarray*}
\mathcal{H} P_c = Q^{-1} \left[ \begin{array}{cc}
(\lambda^2+\xi^2) & 0 \\
0 & -(\lambda^2+\xi^2)
\end{array} \right] Q.
\end{eqnarray*}

Using the above analysis for $L_{-}^{\frac12} L_{+} L_{-}^{\frac12}$, we see that
\begin{eqnarray*}
Q & = & \frac{1}{\sqrt{2}} \left[ \begin{array}{cc}
i (\lambda^2 + \xi^2)^{\frac12} \mathcal{F} L_{-}^{-\frac12} & (\lambda^2 + \xi^2)^{-\frac12} \mathcal{F} L_{-}^{\frac12} \\
-i (\lambda^2 + \xi^2)^{\frac12} \mathcal{F} L_{-}^{-\frac12} &  (\lambda^2 + \xi^2)^{-\frac12} \mathcal{F} L_{-}^{\frac12}
\end{array} \right] \\
& = & T \mathcal{F} P, \\
Q^{-1} & = & \frac{1}{\sqrt{2}} \left[ \begin{array}{cc}
-i L_{-}^{\frac12} \mathcal{F}^*  (\lambda^2 + \xi^2)^{-\frac12} & i L_{-}^{\frac12} \mathcal{F}^* (\lambda^2 + \xi^2)^{-\frac12} \\
L_{-}^{-\frac12} \mathcal{F}^* (\lambda^2 + \xi^2)^{\frac12} & L_{-}^{-\frac12} \mathcal{F}^* (\lambda^2 + \xi^2)^{\frac12}
\end{array} \right] \\
& = & P^{-1} \mathcal{F}^* T^{-1},
\end{eqnarray*}
where 
\begin{eqnarray*}
P = \left[ \begin{array}{cc}
 L_{-}^{-\frac12} & 0 \\
0 &   L_{-}^{\frac12}
\end{array} \right] 
\end{eqnarray*}
and
\begin{eqnarray*}
T = \left[ \begin{array}{cc}
i (\lambda^2 + \xi^2)^{\frac12} & (\lambda^2 + \xi^2)^{-\frac12}  \\
-i (\lambda^2 + \xi^2)^{\frac12}  &  (\lambda^2 + \xi^2)^{-\frac12} 
\end{array} \right] 
\end{eqnarray*}

Note that we have for $\vec{f} = P_c \vec{f}$ 
\begin{eqnarray*}
\mathcal{H} \left[ \begin{array}{c}
f_1 \\
f_2
\end{array} \right] = 
\left[ \begin{array}{c}
i L_{-} f_2 \\
-i L_{+} f_1
\end{array} \right],
\end{eqnarray*}
which is exactly what results from the decomposition.  The resulting integral equation is
\begin{eqnarray*}
\mathcal{H} P_c \vec{f} = \left[ \begin{array}{c}
-i L_{-}^{\frac{1}{2}} \mathcal{F}^* \mathcal{F} L_{-}^{\frac{1}{2}} f_2 \\
i L_{-}^{-\frac{1}{2}} \mathcal{F}^* (\lambda^2 + \xi^2)^2 \mathcal{F} L_{-}^{-\frac{1}{2}} f_1
\end{array} \right].
\end{eqnarray*}

So, since we have a pseudodifferential operator representation of $L_{-}^{\frac{1}{2}}$, we could write $\mathcal{H} P_c$ in terms of an oscillatory integral.

\begin{rem}
We have now made precise the definition
\begin{eqnarray}
\label{eqn:uvrep1}
\tilde{\phi}_{\xi} & = & \left[ \begin{array}{cc}
i u_\xi & v_\xi  \\
-i u_\xi & v_\xi
\end{array} \right] \\
\label{eqn:uvrep2}
& = & \left[ \begin{array}{cc}
i (\xi^2 + \lambda^2) L_{-}^{-\frac12} \tilde{u}_\xi & (\xi^2 + \lambda^2)^{-1} L_{-}^{\frac12} \tilde{u}_\xi  \\
-i (\xi^2 + \lambda^2) L_{-}^{-\frac12} \tilde{u}_\xi & (\xi^2 + \lambda^2)^{-1} L_{-}^{\frac12} \tilde{u}_\xi
\end{array} \right],
\end{eqnarray}
where using the pseudo-differential analysis above, $L_{-}^{\pm \frac12} \tilde{u}_\xi$ is well-defined.
\end{rem}

\begin{proof}[Proof of Theorem \ref{thm:rep}] 

If 
\begin{eqnarray*}
f = \left[ \begin{array}{c}
f_1 \\
f_2 
\end{array} \right] \in \sigma_{ac} (\mathcal{H}),
\end{eqnarray*}
then
\begin{eqnarray*}
\tilde{P} f = \left[ \begin{array}{cc}
 L_{-}^{\frac{1}{2}} & 0 \\
0 & L_{-}^{-\frac{1}{2}}
\end{array} \right] \left[ \begin{array}{c}
f_1 \\
f_2 
\end{array} \right] \in \sigma_{ac} (\tilde{\mathcal{H}}),
\end{eqnarray*}
where 
\begin{eqnarray*}
\tilde{\mathcal{H}} = \left[ \begin{array}{cc}
L_{-}^{\frac{1}{2}} L_{+} L_{-}^{\frac{1}{2}} & 0 \\
0 & L_{-}^{\frac{1}{2}} L_{+} L_{-}^{\frac{1}{2}}
\end{array} \right].
\end{eqnarray*}

Assume $f \in \mathcal{S}$, which we will relax later.  Let $\psi$ be the PDO representation of $\tilde{P}$ and $\Phi_\xi (x)$ is the vector where both elements are the distorted Fourier basis function $\phi_\xi$ for the self-adjoint operator $L_{-}^{\frac{1}{2}} L_{+} L_{-}^{\frac{1}{2}}$.
Then, we have
\begin{eqnarray*}
(\mathcal{G} f)(\xi) & = & T \langle \psi f , \Phi_\xi \rangle \\
& = & T \langle f , \psi^* \Phi_\xi \rangle \\
& = & \langle f , \tilde{\Phi}_\xi \rangle ,
\end{eqnarray*} 
where 
\begin{eqnarray*}
T = \left[ \begin{array}{cc}
i (\lambda^2 + \xi^2)^{\frac{1}{2}} &  (\lambda^2 + \xi^2)^{-\frac{1}{2}} \\
-i (\lambda^2 + \xi^2)^{\frac{1}{2}} &  (\lambda^2 + \xi^2)^{-\frac{1}{2}}
\end{array} \right],
\end{eqnarray*}
and $\tilde{\Phi}_\xi$ is uniquely defined in the sense of distributions as 
\begin{eqnarray*}
P(x,\xi) e^{i x \xi} +\tilde{u}_\xi (x,\xi),
\end{eqnarray*}
where 
\begin{eqnarray*}
\tilde{u} = P(x,D) u_\xi (x),
\end{eqnarray*}
and $\tilde{u} \in \mathcal{S}$.  Then,
\begin{eqnarray*}
(\mathcal{G} f)(\xi) = \int f \tilde{\Phi}_\xi dx.
\end{eqnarray*}

Similarly, we have
\begin{eqnarray*}
(\mathcal{G}^{-1} f)(x) = \int f(\xi) \tilde{\Phi}^{-1}_\xi d\xi,
\end{eqnarray*}
where
\begin{eqnarray*}
\tilde{\Phi}^{-1}_\xi = P^{-1} (T^{-1})^{*} \Phi^*_\xi (x),
\end{eqnarray*}
where $(T^{-1})^{*}$ represents the adjoint of the multiplier matrix $T^{-1}$ above.

The modified Fourier transforms are in fact variations on the expansion involving the matrix $Q$.  
\end{proof}

Note that since $P(T^{-1})^*, TP^* \in S^0$, the regularity properties of $\tilde{\Phi}_\xi$, $\tilde{\Phi}^{-1}$ are the same as those of $\phi_\xi$ as described in \ref{thm:dfbu} with modifications to the explicit formulas.  

Note also the manipulations in the proof of \ref{thm:rep} are valid in the sense of distributions, hence the assumption $f \in \mathcal{S}$.  However, as in the dispersive estimates below, similar estimates are seen to hold for less regular initial data through standard duality and limiting arguments.

\begin{cor}
\label{lin:cor}
As a result of the decomposition, we have a new proof of the fact that
\begin{eqnarray*}
\| P_{c} e^{i t \mathcal{H}} f \|_{L^2} \lesssim \| f \|_{L^2}.
\end{eqnarray*}
\end{cor}

\begin{proof}
This follows simply from mapping properties of pseudodifferential operators and the fact that the self-adjoint distorted Fourier transform is an $L^2$ isometry.
\end{proof}

\begin{rem}
Note that for convenience in terms of defining the resolvent, our result has been proved here only in $\reals^3$.  However, using similar bounds developed in \cite{Ag} for higher dimensional resolvents, we expect that a result similar to that of \ref{thm:rep} holds in all dimensions and as a result similar estimates will follow below.  The main difficulties presented would be a thorough discussion of the spectrum of $\mathcal{H}$ as some of the known numerical techniques are unique to $\reals^3$.  
\end{rem}

\section{Time Decay}
\label{lin:tdecay}

Using our distorted Fourier basis, we have that a solution to the problem
\begin{eqnarray}
\label{eqn:intrep}
e^{i \mathcal{H} t} P_c \phi = Q^{-1} e^{i t W} Q \phi,
\end{eqnarray}
for
\begin{eqnarray*}
W = \left[ \begin{array}{cc}
(\lambda^2+\xi^2) & 0 \\
0 & -(\lambda^2+\xi^2)
\end{array} \right].
\end{eqnarray*}
The structure on $Q$ allows us to do oscillatory integration in order to study the properties of $e^{i \mathcal{H} t}$.  First of all, we prove Theorem \ref{thm:tdec}.  We will fix the notation $K = K^+$.

\begin{proof}[Proof of \ref{thm:tdec}]

Using matrix notation, we have
\begin{eqnarray*}
\{ \mathcal{G} \vec{\psi} \} (\xi) = \int \tilde{\phi}_{\xi} (x) \vec{\psi}(x) dx,
\end{eqnarray*}
where
\begin{eqnarray*}
\vec{\psi} (x) = \left[ \begin{array}{c}
\psi_1 \\
\psi_2
\end{array} \right],
\end{eqnarray*}
and $\tilde{\phi}_\xi$ is given by \eqref{eqn:uvrep1}.

Looking at the integral representation, we have
\begin{eqnarray*}
e^{i t \mathcal{H}} P_{c} \vec{\psi} (x) = \int_\xi \tilde{\phi}_{\xi}^{-1} (x) e^{itW} \int_y \tilde{\phi}_{\xi} (y) \vec{\psi} (y) dy d\xi.
\end{eqnarray*}
Let $\chi \in C^\infty_c$, be a smooth, cut-off function chosen such that the iteration techniques in Theorem \ref{thm:dfbu} hold for $\xi \in \reals^d \setminus \spt (\chi)$.  Then, take
\begin{eqnarray}
\label{eqn:stph1}
e^{i t \mathcal{H}} P_{c} \vec{\psi} (x) & = & \int_\xi \chi(\xi) \tilde{\phi}_{\xi}^{-1} (x) e^{itW} \int_y \tilde{\phi}_{\xi} (y) \vec{\psi} (y) dy d\xi \\
\label{eqn:stph2}
& + & \int_\xi [1-\chi(\xi)] \tilde{\phi}_{\xi}^{-1} (x) e^{itW} \int_y \tilde{\phi}_{\xi} (y) \vec{\psi} (y) dy d\xi .
\end{eqnarray}
Hence, we must bound
\begin{eqnarray*}
I & = & \int_\xi e^{i x \xi} e^{\pm it (\xi^2 + \lambda^2)} \int_y e^{-i y \xi} \psi (y) dy d\xi, \\
II & = & \int_\xi [\chi(\xi) + (1-\chi(\xi))]  g^{-1}_\xi (x) e^{ \pm it (\xi^2 + \lambda^2)} \int_y e^{-i y \xi} \psi (y) dy d\xi, \\
III & = & \int_\xi [\chi(\xi) + (1-\chi(\xi))]   e^{i x \xi} e^{ \pm it (\xi^2 + \lambda^2)} \int_y g_\xi (y) \psi (y) dy d\xi, \\
IV & = & \int_\xi [\chi(\xi) + (1-\chi(\xi))]   g^{-1}_\xi (x) e^{ \pm it (\xi^2 + \lambda^2)} \int_y g_\xi (y) \psi (y) dy d\xi.
\end{eqnarray*}

From henceforward, we work only with the term
\begin{eqnarray*}
\frac{e^{ i |x| |\xi|}}{|x|}
\end{eqnarray*}
from $g_\xi$, as the analysis for the exponentially decaying term will follow using simpler versions of the methods for this case.  Many of the techniques used are developed from the presentation in \cite{S1}.  The challenge lies mostly in that $\partial_\xi^\alpha |\xi|$ is not bounded near $0$ for $|\alpha| \geq 2$.  Thus, we must be careful near the origin using stationary phase arguments since error terms require a minimum of two derivatives.  A discussion of stationary phase complete with proofs is given in \cite{EvZw} or \cite{Ste}.  Take the integral,
\begin{eqnarray*}
\mathcal{I} = \int h(x) e^{i \tau P(x)} dx,
\end{eqnarray*}
where $h (x) \in C^\infty_c$, $P(x) \in C^\infty$.  Assume that $\partial_x P(0) = 0$ and $\partial_x^2 P(0) \neq 0$.  Then, the principle of stationary phase gives
\begin{eqnarray*}
\mathcal{I} \sim \tau^{-\frac{d}{2}} \sum_{j=0}^\infty a_j \tau^{-j},
\end{eqnarray*}
where the asymptotic terms in the stationary phase expansion are given by
\begin{eqnarray*}
a_j = L^j h(0),
\end{eqnarray*}
for $L^j$ an order $2j$ differential operator as discussed in \cite{EvZw}.

Equation $I$ is bounded using standard techniques of contour integration from the Linear Schr\"odinger equation.  In particular, we have
\begin{eqnarray*}
\| I \|_{L^\infty} \lesssim t^{-\frac{d}{2}} \| \vec{\psi} \|_{L^1}.
\end{eqnarray*}

Before we investigate further, we recall some properties of the functions $\partial^\alpha_{\xi_0} g_{\xi_0}$.  From the expression \eqref{eqn:dfbg} for $g_{\xi_0}$, we know that
\begin{eqnarray*}
g_{\xi_0} = K*(\tilde{V} (x,\xi_0) e^{ix \cdot \xi_0}) + K*(\tilde{V}(x,D) g_{\xi_0}),
\end{eqnarray*}
where
\begin{eqnarray*}
\tilde{V} g_{\xi_0} = \tilde{V}_1 (I - \tilde{V}_2 K \tilde{V}_1)^{-1} \tilde{V}_2 (K * (\tilde{V} (x,\xi_0) e^{i \cdot \xi_0})).
\end{eqnarray*}
From Fredholm Theory and the spectral assumptions on $\mathcal{H}$, $(I - P_{\xi_0})^{-1}$ is well-defined, hence we can show that $\tilde{V} g_{\xi_0}$ is smooth in $|\xi_0|$ and $\xi_0$.  Also, $K$ is smooth with respect to $|\xi|$, $V_\xi e^{i x \xi}$ is smooth with respect to $\xi$.  As a result, as proved in Theorem \ref{thm:dfbu}, $g_{\xi_0} = K*f_0$ where $f_0$ depends smoothly on $|\xi|$ and $\xi$.  Therefore, for $\xi$ near $0$, we can take up to $3$ derivatives of the standard stationary phase operator 
$$L = \frac{\xi}{2it |\xi|^2} \partial_\xi$$
before we lose integrability in $\xi$.  For $\xi_0$ large enough, from Theorem \ref{thm:dfbu}, we have
\begin{eqnarray*}
g_{\xi_0} = K * f,
\end{eqnarray*}
where
\begin{eqnarray*}
f = e^{i x \cdot \xi_0} f_0(x,\xi_0),
\end{eqnarray*}
where $f_0(x,\xi_0)$ behaves like a symbol in $S^{2}$.

For \eqref{eqn:stph2}, we use the principle of nonstationary phase and the principle of stationary phase in different regions.  We have
\begin{eqnarray*}
\int_\xi [1-\chi(\xi)] \tilde{\phi}_{\xi}^{-1} (x) e^{it W} \int_y \tilde{\phi}_{\xi} (y) \vec{\psi} (y) dy d\xi,
\end{eqnarray*}
where $1-\chi$ is supported away from $0$.  In particular, we have integrals of the type
\begin{eqnarray*}
\int_\xi [1-\chi(\xi)] (e^{i x \xi} + \tilde{g}_\xi (x)) e^{\pm it (\xi^2 + \lambda^2)} \int_y (e^{-i y \xi} + g_\xi (y)) \psi (y) dy d\xi,
\end{eqnarray*}
where $g$ and $\tilde{g}$ are of the same form described above.
Hence, we must bound the following
\begin{eqnarray*}
II^* & = & \int_\xi [1-\chi(\xi)]  \bar{g}_\xi (x) e^{ it (\xi^2 + \lambda^2)}  e^{-i y \xi} \psi (y)  d\xi, \\
III^* & = & \int_\xi [1-\chi(\xi)] e^{i x \xi} e^{ it (\xi^2 + \lambda^2)}  g_\xi (y) \psi (y)  d\xi, \\
IV^* & = & \int_\xi [1-\chi(\xi)] \bar{g}_\xi (x) e^{ it (\xi^2 + \lambda^2)}  g_\xi (y) \psi (y)  d\xi.
\end{eqnarray*}
The bounds for $e^{ -it (\xi^2 + \lambda^2)}$ will follow through similar arguments.

For integrals of type $II^*$, we have oscillatory integrals of the form
\begin{eqnarray}
\label{eqn:dec2}
\int e^{i (-|z| |\xi_0| - (x-z) \xi_0 + t \xi_0^2 - y \cdot \xi_0)} \frac{\bar{f}_{0}(x-z, \xi_0)}{|z|(\lambda^2 + \xi_0^2)} d\xi_0.
\end{eqnarray}
Looking at the phase function, we have 
\begin{eqnarray*}
\phi (\xi_0) & = & -|z| |\xi_0| - (x-z) \xi_0 + t \xi_0^2 - y \cdot \xi_0, \\
\nabla_{\xi_0} \phi (\xi_0) & = &   2 t \xi_0 -(x-z+ y) - |z| \frac{\xi_0}{|\xi_0|}, \\
\nabla^2_{\xi_0} \phi (\xi_0) & = & 2t I_d - \frac{|z|}{|\xi_0|} (I_d - \frac{\xi_0 \otimes \xi_0}{|\xi_0|^2}).
\end{eqnarray*}
If we restrict $\xi_0$ to a region such that 
\begin{eqnarray*}
|\xi_0| \geq \frac{ |z-y-x| + |z| }{2t} + 1,
\end{eqnarray*} 
then $\phi (\xi_0)$ has no critical points.  As a result, we can use the principle of non-stationary phase on this region with the decay properties of the function $f_0$ to see we have decay like $t^{-N}$ for any $N$.  

Let us hence assume that we are restricted a region
\begin{eqnarray*}
|\xi_0| \leq \frac{ |z+y-x| + |z| }{2t} + 1,
\end{eqnarray*} 
so $\phi$ has at least one critical point.  In fact, the critical point occurs where
\begin{eqnarray}
\label{eqn:cp}
\xi_0 \left( 2t - \frac{|z|}{|\xi_0|} \right) = z+y-x.
\end{eqnarray}
For $|z| - |z-y-x|<  0$, we have only
\begin{eqnarray}
\label{eqn:cpabs}
|\xi_0| = \frac{|z-y-x| + |z| }{2t}.
\end{eqnarray}
Otherwise, we have also
\begin{eqnarray}
\label{eqn:cpabs1}
|\xi_0| = \frac{ |z| - |z-y-x| }{2t}.
\end{eqnarray}

As a result, all critical points occur on one of two spheres.  Using \eqref{eqn:cpabs} and \eqref{eqn:cpabs1}, we have that if $z$, $y$ and $x$ are such that a critical point exists, that critical point is unique.  Hence, we can define a cut-off function $\chi_{x,y,z} \in C^\infty_c (\reals^d)$ such that 
\begin{eqnarray*}
\chi_{x,y,z} (\xi) = \left\{ \begin{array}{c}
1 \ \text{for} \ |\xi_0| \leq \frac{ |z-y-x| + |z| }{2t} + \frac{M}{4}, \\
0 \ \text{for} \ |\xi_0| \geq \frac{ |z-y-x| + |z| }{2t} + \frac{M}{2}.
\end{array} \right.
\end{eqnarray*}

Let us assume that a critical point exists, say $\xi_0^c$.  If $|\xi_0^c| < \frac{|z|}{2t}$, the Hessian matrix is at least of rank $1$ as $\xi \otimes \xi$ is a rank $1$ matrix.  So, there is at least one nondegenerate direction for $\xi$.  After making an orthogonal change of coordinates bringing that nondegenerate direction to $\xi_1$, using stationary phase on $\reals$, we have decay of the form
\begin{eqnarray*}
\| \text{\eqref{eqn:dec2}} \|_{L^\infty} \lesssim t^{-\frac{1}{2}}.
\end{eqnarray*}
However, in the integral, we have $\frac{1}{|z|} < \frac{1}{|\xi_0|t}$ and $|\xi_0| > c > 0$, so using the decay of $f_0$ in $z$, the overall decay is once again
\begin{eqnarray*}
\| \text{\eqref{eqn:dec2}} \|_{L^\infty} \lesssim t^{-\frac{d}{2}},
\end{eqnarray*}
where the error is bounded by 
\begin{eqnarray*}
\sup_{|\alpha| \leq 2} | \partial_{\xi_1}^\alpha \frac{f_{0}(x-z, \xi_0)}{(\lambda^2 + \xi_0^2)} |  .
\end{eqnarray*}
As $f_0 \in S^2$, this follows easily.

For $|\xi_0^c| > \frac{|z|}{2t}$, the Hessian is nondegenerate. 
We can thus apply stationary phase in $\xi$ to get decay of the form 
\begin{eqnarray*}
\| \text{\eqref{eqn:dec2}} \|_{L^\infty} \lesssim t^{-\frac{d}{2}},
\end{eqnarray*}
where we have once again used the regularity of $f_0$ is $x$ and $\xi$.  
Then, given the uniform decay of $f_0$ and boundedness in $y$ and $x$, we have uniform boundedness with decay of type $t^{-\frac{d}{2}}$.  The result for type $III^*$ follows similarly.   

The analysis for oscillatory integrals of type $IV^*$ is similar in that the phase function becomes
\begin{eqnarray*}
\phi (\xi_0) & = & |z| |\xi_0| + (x-z) \xi_0 + t \xi_0^2 - |z_0| |\xi_0| - (y-z_0) \xi_0, \\
\nabla_{\xi_0} \phi (\xi_0) & = & 2 t \xi_0 + (x-z) - (y-z_0) + (|z|-|z_0|) \frac{\xi_0}{|\xi_0|}, \\
\nabla^2_{\xi_0} \phi (\xi_0) & = & 2t I_d + \frac{|z|-|z_0|}{|\xi_0|} (I_d - \frac{\xi_0 \otimes \xi_0}{|\xi_0|^2} ).
\end{eqnarray*}
Hence, where critical points exist, we split up the regions of integration into $|\xi_0| > \frac{|z|-|z_0|}{2t }$ and $|\xi_0| < \frac{|z|-|z_0|}{2 t}$.  Once again, we have stationary phase in full on the first region and stationary phase in at least one direction, coupled with the fact that $\frac{1}{|z|} < \frac{1}{2 |\xi_0| t}$.  Away from the critical points, we once again apply non-stationary phase. 

Let us now analyze \eqref{eqn:stph1}.  In particular, we have integrals of the type
\begin{eqnarray*}
\int_\xi [\chi(\xi)] (e^{i x \xi} + \bar{g}_\xi (x)) e^{it (\xi^2 + \lambda^2)} \int_y (e^{-i y \xi} + g_\xi (y)) \psi (y) dy  d\xi.
\end{eqnarray*}
Thus, we have to bound
\begin{eqnarray*}
II^{**} & = & \int_\xi [\chi(\xi)] g^{-1}_\xi (x) e^{ it (\xi^2 + \lambda^2)}  e^{-i y \xi}   d\xi, \\
III^{**} & = & \int_\xi [\chi(\xi)] e^{i x \xi} e^{ it (\xi^2 + \lambda^2)}  g_\xi (y) d\xi, \\
IV^{**} & = & \int_\xi [\chi(\xi)] g^{-1}_\xi (x) e^{  it (\xi^2 + \lambda^2)}  g_\xi (y)  d \xi.
\end{eqnarray*}
Once again, the bounds for $e^{-it (\xi^2 + \lambda^2)}$ will follow from similar techniques.

For integrals of type $II^{**}$ and $III^{**}$, we have an oscillatory integral of the form
\begin{eqnarray}
\label{eqn:dec1}
\int e^{i (-|x| |\xi_0| + t \xi_0^2 - y \xi_0)} \frac{f_{0} (x-z)}{|x|} d\xi_0.
\end{eqnarray}
The phase function is 
\begin{eqnarray*}
\phi (\xi_0) & = & -|x| |\xi_0| + t \xi_0^2 - y \xi_0, \\
\nabla_{\xi_0} \phi (\xi_0) & = & 2 t \xi_0 - y - |x| \frac{\xi_0}{|\xi_0|}, \\
\nabla^2_{\xi_0} \phi (\xi_0) & = & 2t I_d - \frac{|x|}{|\xi_0|} (I_d - \frac{\xi_0 \otimes \xi_0}{|\xi_0|^2}).
\end{eqnarray*}
Let us begin with an integral of type $II^{**}$.  After making the orthogonal change of coordinates $\xi_1 \to \frac{y}{|y|}$ and moving to polar coordinates in $\xi$, we need to bound
\begin{eqnarray*}
\label{eqn:II}
\| II^{**} \|_{L^\infty} & \lesssim & \| \int_0^{2 \pi} \int_0^\pi \int_0^\infty \int \chi(r) \frac{e^{-ir|z_0|}}{|z_0|} e^{itr^2} e^{-ir|y| \cos(\theta)} \\
& \times &  f_0 (x-z_0,r \sin (\theta) \cos (\phi), r \sin (\theta) \sin (\phi), r \cos (\theta), r)  r^2 \sin (\theta) dr d\theta d\phi \|_{L^\infty}.
\end{eqnarray*}
If we Taylor expand $f_0$ in terms of $(r \sin (\theta) \cos (\phi), r \sin (\theta) \sin (\phi), r \cos (\theta), r)$, then we can integrate in $\phi$.  In which case, all terms in the expansion with odd powers of $\cos (\phi)$ or $\sin (\phi)$ vanish under integrating out, leaving us with a function of the form
\begin{eqnarray*}
\tilde{f}_0 ( r^2 \cos^2 ( \theta ), r ).
\end{eqnarray*}
Integrating by parts in $\theta$, we have
\begin{eqnarray*}
\| II^{**} \|_{L^\infty} & \lesssim & \| \int_0^{2 \pi} \int_0^\pi \int_0^\infty \chi(r) \frac{e^{-ir|z_0|}}{|z_0|} e^{itr^2} e^{-ir|y| \cos(\theta)} \\
& \times &  \tilde{f}_0 (x-z_0,r \cos^2 ( \theta ), r)  r^2 \sin (\theta)  dr d\theta d\phi \|_{L^\infty} \\
& \lesssim & \| \int_0^\infty \chi(r) \frac{e^{-ir|z_0|}}{|z_0|} e^{itr^2} \frac{\sin(r|y|)}{|y|} \\
& \times & \tilde{f}_0 (x-z_0, r)  r  dr \|_{L^\infty} \\
& + & \| \int_0^\pi \int_0^\infty \chi(r) \frac{e^{-ir|z_0|}}{|z_0|} e^{itr^2} \frac{e^{-ir|y| \cos(\theta)}}{|y|} \\
& \times & \partial_\theta \tilde{f}_0 (x-z_0, r \cos^2 (\theta), r)  r dy dz_0  dr d\theta \|_{L^\infty} \\
& = & \| \int_0^\infty \chi(r) \frac{e^{-ir|z_0|}}{|z_0|} e^{itr^2} \frac{\sin(r|y|)}{|y|} \\
& \times & \tilde{f}_0 (x-z_0, r)  r  dr \|_{L^\infty} \\
& + & \| \int_0^\pi \int_0^\infty \chi(r) \frac{e^{-ir|z_0|}}{|z_0|} e^{itr^2} \frac{\sin(-ir|y| \cos(\theta))}{|y|} \\
& \times & \partial_\theta \tilde{f}_0 (x-z_0, r \cos^2 (\theta), r)  r  dr d\theta \|_{L^\infty},
\end{eqnarray*}
since for $n$ odd, 
\begin{eqnarray*}
\int_{-1}^1 e^{i \mu x} x^n dx = i \int_{-1}^1 \sin (\mu x) x^n dx.
\end{eqnarray*}
Note that the boundedness in $y$ and $\theta$ is hence maintained after the integration by parts.

Let us extend the region of integration in $r$ to $\reals$.  Due to the nature of the oscillatory functions involved, we experience no loss in doing so.  Then, using the linear Schr\"odinger equation dispersion, we have
\begin{eqnarray*}
\| II^{**} \|_{L^\infty} & \lesssim & \frac{1}{t} \| \int_{-\infty}^{\infty} e^{itr^2}   \\
& \times & \partial_r \left[ \frac{e^{-ir|z_0|}}{|z_0|} \frac{e^{-ir|y|}-e^{ir|y|}}{|y|} \chi(r) f_0 (x-z_0,r)  r \right] dr \|_{L^\infty} \\
& \lesssim & \frac{1}{t^{\frac{3}{2}}} \| \int_{-\infty}^{\infty} |\int \frac{1}{|z_0||y|} \\
& \times & [ \mathcal{F}^{-1} \left[ \chi(r) f_0 (x-z_0, r)  r \right] (u + |z_0| + |y|) (x-z_0) \\
& - & \mathcal{F}^{-1} \left[ \chi(r) f_0 (x-z_0, r)  r \right] (u + |z_0| - |y|) (x-z_0) dy | du  \|_{L^\infty}.
\end{eqnarray*}
From the estimate
\begin{eqnarray*}
\| \hat{u} \|_{L^1} \lesssim \sup_{|\alpha| \leq d+1} \| \partial^\alpha u \|_{L^1},
\end{eqnarray*}
coupled with the facts that $\chi \in C^\infty_0$, $f_0 \in C^\infty_r$, and $f_0$ is rapidly decaying in $x$, we have
\begin{eqnarray*}
\| II^{**} \|_{L^\infty} \lesssim \frac{1}{t^{\frac{3}{2}}} \| f \|_{L^1}.
\end{eqnarray*}

For the integrals of type $III^{**}$, we immediately apply the linear Schr\"odinger estimate to get
\begin{eqnarray*}
\| III^{**} \|_{L^\infty} & \lesssim & \frac{1}{t^{\frac{3}{2}}} \int | \int \chi (\xi) \frac{e^{i |y-x_1| |\xi|} }{|y-x_1|} e^{i y \xi} f_0 (x_1, \xi, |\xi|) f(y) dx_1 d\xi | dy \\
& \lesssim & \frac{1}{t^{\frac{3}{2}}},
\end{eqnarray*}
using once again the smoothness and decay of $\chi$, $f_0$.

The analysis for oscillatory integrals of type $IV^{**}$ is similar to that for type $II^{**}$, except now we have no $\theta$ dependence in the phase.  Thus, we have phase functions of the form
\begin{eqnarray*}
\phi (\xi_0) & = & -|x| |\xi_0| + t \xi_0^2 + |y| |\xi_0|, \\
\nabla_{\xi_0} \phi (\xi_0) & = & 2 t \xi_0 + (|y|-|x|) \frac{\xi_0}{|\xi_0|}, \\
\nabla^2_{\xi_0} \phi (\xi_0) & = & 2t I_d + \frac{|y|-|x|}{|\xi_0|} (I_d - \frac{\xi_0 \otimes \xi_0}{|\xi_0|^2}).
\end{eqnarray*}
At this point, it becomes convenient to move to polar coordinates in $\xi$.  As a result, we have
\begin{eqnarray*}
\| IV^{**} \|_{L^\infty} & \lesssim & \int_0^{2\pi} \int_0^\pi \int_0^\infty \chi(r) \frac{e^{-ir|z_0|}}{|z_0|} e^{itr^2} \frac{e^{ir|z_1|}}{|z_1|} \\
& \times & f_0 (y-z_0,r \sin (\theta) \cos (\phi), r \sin (\theta) \sin (\phi), r \cos (\theta), r)  \\
& \times & \tilde{f}_0 (x-z_1,r \sin (\theta) \cos (\phi), r \sin (\theta) \sin (\phi), r \cos (\theta), r) \\
& \times & r^2 \sin (\theta) dr d\theta d\phi.
\end{eqnarray*}
Hence, we can first extend the interval of integration in $r$ to $\reals$, then immediately integrate by parts in $r$ to gain a factor of $\frac{1}{t}$.  We once again apply the linear Schr\"odinger dispersive estimate to get 
\begin{eqnarray*}
\| IV^{**} \|_{L^\infty} \lesssim \frac{1}{t^{\frac{3}{2}}} \| f \|_{L^1}.
\end{eqnarray*}

Combining the above results, we have
\begin{eqnarray*}
\| \text{\eqref{eqn:stph1}} \|_{L^\infty} \leq t^{-\frac{d}{2}} \| \psi \|_{L^1}
\end{eqnarray*}
and
\begin{eqnarray*}
\| \text{\eqref{eqn:stph2}} \|_{L^\infty} \leq t^{-\frac{d}{2}} \| \psi \|_{L^1}.
\end{eqnarray*}
Hence, the theorem follows.

\end{proof}

We now proceed to prove Theorem \ref{thm:wlindec}.

\begin{proof}[Proof of \ref{thm:wlindec}]

We proceed similarly to the proof of Theorem \ref{thm:tdec}, except now we must bound the following:
\begin{eqnarray*}
I & = & |e^{-c|x|} \int_\xi \chi(\xi)  \phi^{-1}_\xi (x) e^{ \pm it (\xi^2 + \lambda^2)} \int_y \phi_{\xi} (y) \psi (y) dy d\xi|, \\
II & = & |e^{-c|x|} \int_\xi [1-\chi(\xi)]  \phi^{-1}_\xi (x) e^{ \pm it (\xi^2 + \lambda^2)} \int_y \phi_{\xi} (y) \psi (y) dy d\xi|.
\end{eqnarray*}

For $II$, we look at oscillatory integrals of the form
\begin{eqnarray*}
\int_\xi [1-\chi(\xi)] (e^{i x \cdot \xi} + g^{-1}_\xi (x)) e^{\pm it (\xi^2 + \lambda^2)} \int_y (e^{i y \cdot \xi} + g_\xi (y)) \psi (y) dy d\xi.
\end{eqnarray*}
Motivated by the principle of stationary phase in \cite{EvZw}, define the operator
\begin{eqnarray*}
L = \frac{\langle \xi, \partial_\xi \rangle}{\pm 2 |\xi|^2 it}.
\end{eqnarray*}
Considering the phase function as $\phi(\xi) = t \xi^2$, it is clear
\begin{eqnarray*}
L e^{i \phi(\xi)} = e^{i \phi(\xi)}.
\end{eqnarray*}
Then, let us take $L^M e^{i \phi(\xi)}$ in $II$ and integrate by parts.  Note, on the support of $1 - \chi(\xi)$, $\xi/|\xi|^2$  is a bounded multiplier.  A calculation shows
\begin{eqnarray*} 
|\partial_\xi g_\xi| & \leq & |\int \frac{\xi}{|\xi|} e^{i |x-y| |\xi|} e^{i y \xi} f_0 (y,\xi) dy | \\
& + &|\int \frac{\xi}{\sqrt{\xi^2+\lambda^2}} e^{- |x-y| \sqrt{\xi^2+\lambda^2}} e^{i y \xi} f_0 (y,\xi) dy | \\
& + & | \int \frac{e^{i |x-y| |\xi|} - e^{- |x-y| \sqrt{\xi^2+\lambda^2}}}{|x-y|} e^{i y \cdot \xi} y f_0 (y,\xi) dy | \\
& + & | \int \frac{e^{i |x-y| |\xi|} - e^{- |x-y| \sqrt{\xi^2+\lambda^2}}}{|x-y|} e^{i y \cdot \xi} \partial_{\xi} f_0 (y,\xi) dy| \\
& \lesssim &  \int (\langle x \rangle + \langle y \rangle) |f_0 (y,\xi)| dy + \int | \partial_{\xi} f_0 (y,\xi) | dy.
\end{eqnarray*}
Using the regularity of $f_0$ in $y$ and $\xi_0$ and continuing this calculation for $\partial^M_\xi g_\xi$,  by applying the decay results from similar terms in \ref{thm:tdec} we see
\begin{eqnarray*}
\| e^{-c|x|} \int_\xi (1-\chi(\xi)) \tilde{\phi}_{\xi}^{-1} (x) e^{itW} \int_y \tilde{\phi}_{\xi} (y) \vec{\psi} (y) dy d\xi \|_{L^\infty} \lesssim t^{ - \frac{d}{2}-M} \| \psi \|_{L^{1,M}}.
\end{eqnarray*}

Now, for $I$, we need to bound
\begin{eqnarray*}
\int_\xi [\chi(\xi)] (e^{i x \cdot \xi} + g^{-1}_\xi (x)) e^{it (\xi^2 + \lambda^2)} (e^{i y \cdot \xi} + g_\xi (y))  d\xi.
\end{eqnarray*}
It is here our moments conditions become necessary.  We wish to proceed similarly to case $II$, but now $\frac{\xi}{|\xi|^2}$ is a singular multiplier.  In fact, note that after integration by parts $M$ times, the leading order operator will be on the order of $|\xi|^{-2M}$.
As a result, we arrive at the $2M$ moments conditions in \eqref{eqn:mom1}.
We have a gain in time decay using integration by parts in $L$, and since
\begin{eqnarray*}
L_j \vec{g} (0) = 0,
\end{eqnarray*}
for $j = 1,\dots,2M$, there still no singularities near $\xi = 0$ where
\begin{eqnarray}
\vec{g} (\xi) = \chi(\xi) \tilde{\phi}_{\xi}^{-1} (x) \int_y \tilde{\phi}_{\xi} (y) \vec{f} (y) dy,
\end{eqnarray}
and $L_j$ is the order $2j$ differential operator resulting from the stationary phase-like arguments.  Now, again we can apply the applicable results on oscillatory integrals of the terms $II^{**}$, $III^{**}$ and $IV^{**}$ from the proof of Theorem \ref{thm:tdec} with new functions $f^M_0$
\begin{eqnarray*}
f^M_0 (x-z,y,\xi,|\xi|) = |x|^{M_1} y^{M_2} m_{M_1,M_2} (\xi,|\xi|) L^{M_3} f_0(x-z,\xi,|\xi|)
\end{eqnarray*}
defined on the support of $\chi (\xi)$ where $M_1 + M_2 + M_3 = 2M$.  Using the moments conditions and the weighted integrability of $f$, the argument proceeds precisely as that near $\xi = 0$ for the unweighted time decay case. Hence, under our assumptions we have
\begin{eqnarray*}
\| e^{-c|x|} \int_\xi \chi(\xi) \tilde{\phi}_{\xi}^{-1} (x) e^{itW} \int_y \tilde{\phi}_{\xi} (y) \vec{f} (y) dy d\xi \|_{L^\infty} \lesssim t^{-\frac{d}{2} - M}.
\end{eqnarray*}

\end{proof}

\begin{rem}
In turn, \eqref{eqn:mom1} becomes our moments condition for the function space $\mathcal{P}^A_2$ as defined by
\begin{eqnarray*}
\mathcal{P}^A_2 = \{ \phi \in P_c \mathcal{H} | \| \phi \|_{H^A} < \infty, \ \| |x|^A \phi \|_{L^2} < \infty,  \ \text{condition \ref{eqn:mom1} is satisfied for $j \leq A$} \},
\end{eqnarray*}
with norm
\begin{eqnarray*}
\| \phi \|_{\mathcal{P}^A_2} = \left( \| \phi \|_{H^A}^2 + \| |x|^A \phi \|_{L^2}^2 \right)^{\frac{1}{2}}.
\end{eqnarray*}
These function spaces will be used in \cite{Mnl} in order to find stable perturbations of minimal mass solitons.
\end{rem}

\section{Dispersive Estimates}
\label{lin:disp}

From \cite{W1} or \cite{Mspec}, we have $H^1 = M \otimes S$ where $M$ is $2d+4$ dimensional set of functions that span the $4$th order generalized null space at $0$ and $S$ is the continuous spectrum.  

Since $M$ is spanned by functions with exponential decay, we have for $\phi \in M$
\begin{eqnarray*}
\| e^{it \mathcal{H}} \phi \|_{H^1} \leq C (1 + |t|^3) \int e^{-c|x|} |\phi(x)| dx,
\end{eqnarray*}
where $c$ is determined by the exponential decay of all functions in $M$.  

Now, from \cite{ES1} and \ref{lin:rep} we have for $\phi \in S$,
\begin{eqnarray}
\label{eqn6:1}
\| e^{it \mathcal{H}} \phi \|_{L^2} \leq C \| \phi \|_{L^2}.
\end{eqnarray}

\begin{lem}
Given Equation \eqref{eqn6:1}, we have
\begin{eqnarray*}
\| e^{it \mathcal{H}} \phi \|_{H^1} \leq C \| \phi \|_{H^1}.
\end{eqnarray*}
\end{lem}

\begin{proof}
For $\phi \in S$, we have
\begin{eqnarray*}
\| e^{it \mathcal{H}} \phi \|_{H^2} & \leq & \| \mathcal{H} e^{it \mathcal{H}} \phi \|_{L^2} + C \| e^{it \mathcal{H}} \phi \|_{L^2} \\
& \leq &  \| e^{it \mathcal{H}}  \mathcal{H} \phi \|_{L^2} + C \| e^{it \mathcal{H}} \phi \|_{L^2} \\
& \leq & \| \mathcal{H} \phi \|_{L^2} + C \| e^{it \mathcal{H}} \phi \|_{L^2} \\
& \leq & \| \phi \|_{H^2} + C \| \phi \|_{L^2} \\
& \leq & C \| \phi \|_{H^2}.  
\end{eqnarray*}
Hence, the result follows from interpolation.
\end{proof}

In order to push through the contraction argument, we need various dispersive estimates from \cite{BW}.  We present the proofs here.

\begin{thm}[Erdogan-Schlag,Bourgain]
\label{thm:disp1}

Let $P_c$ and $P_d$ be projections onto the continuous and discrete spectrum of $\mathcal{H}$ respectively.  Then,
\begin{eqnarray*}
(i) \ \| e^{it \mathcal{H}} P_c \phi \|_{H^1} & \leq & C \| \phi \|_{H^1} \\
(ii) \ \| e^{it \mathcal{H}} (P_c\phi) \|_{H^s} & \leq & C \| \phi \|_{H^s} \\
(iii) \ \| e^{it \mathcal{H}} (P_d\phi) \|_{H^s} & \leq & C (1 + |t|^3) \int e^{-c|x|} |\phi(x)| dx \\
(iv) \ \| |x|^\alpha e^{it \mathcal{H}} (P_c \phi) \|_{L^2} & \leq & C( \| |x|^\alpha \phi \|_{L^2} + (1+|t|^\alpha) \| \phi \|_{H^\alpha}) \\
(v) \ \| |x|^\alpha e^{it \mathcal{H}} (P_d \phi) \|_{L^2} & \leq & C (1+|t|^3) \int |\phi| e^{-c|x|} dx.
\end{eqnarray*}
\end{thm}

\begin{proof}
Estimate $(iii)$ follows from the discrete spectral decomposition into a $4$ dimensional generalized null space.  The exponential decay is apparent from the properties of the eigenfunctions.  Estimate $(v)$ follows similarly. 

For $\phi \in \sigma_{ac} (\mathcal{H})$, we have from Section \ref{lin:tdecay} or \cite{ES1} that
\begin{eqnarray*}
\| e^{it \mathcal{H}} P_c \phi \|_{L^2} & \leq & C \| \phi \|_{L^2}.
\end{eqnarray*}
For $\phi \in \sigma_{ac} (\mathcal{H})$, we have
\begin{eqnarray*}
\| e^{it \mathcal{H}} \phi \|_{H^2} & \leq & \| \mathcal{H} e^{it \mathcal{H}} \phi \|_{L^2} + C \| e^{it \mathcal{H}} \phi \|_{L^2} \\
& \leq &  \| e^{it \mathcal{H}}  \mathcal{H} \phi \|_{L^2} + C \| e^{it \mathcal{H}} \phi \|_{L^2} \\
& \leq & \| \mathcal{H} \phi \|_{L^2} + C \| e^{it \mathcal{H}} \phi \|_{L^2} \\
& \leq & \| \phi \|_{H^2} + C \| \phi \|_{L^2} \\
& \leq & C \| \phi \|_{H^2}.  
\end{eqnarray*}
This gives $(i)$.  A similar argument shows
\begin{eqnarray*}
\| e^{it \mathcal{H}} \phi \|_{H^{2s+1}} \lesssim \| \phi \|_{H^{2s+1}} + \| e^{it \mathcal{H}} \phi \|_{H^{2s-1}}.
\end{eqnarray*}
Thus, by induction, we have $(ii)$ for all positive integers $s$ and hence by interpolation all $s>0$.  

Let $\phi \in \sigma_{ac} (\mathcal{H})$ and $u = e^{it \mathcal{H}} \phi$.  Then, since
\begin{eqnarray*}
i v_t - \mathcal{H} v = 0,
\end{eqnarray*}
then
\begin{eqnarray*}
\frac{d}{dt} \int |x|^{2 \alpha|} | v(x,t)|^2 dx & = & 2 \Re \langle |x|^{2 \alpha} v, v_t \rangle \\
& = & 2 \Im \langle |x|^{2 \alpha} v, \mathcal{H} v \rangle \\
& = & 2 \Im \langle |x|^{2 \alpha} v, \Delta v \rangle + O \left( \int |v|^2 e^{-c|x|} \right) \\
& \lesssim & \int |x|^{2 \alpha -1} |v| |\nabla v| dx + \| v \|_2^2.
\end{eqnarray*}
Using the following interpolation inequality
\begin{eqnarray*}
\| |x|^{\alpha - \gamma} | D^{\gamma} v| \|_{L^2} \leq \| |x|^\alpha v \|_{L^2}^{1-\frac{\gamma}{\alpha}} \| v \|_{H^\alpha}^{\frac{\gamma}{\alpha}},
\end{eqnarray*}
we have
\begin{eqnarray*}
\int |x|^{2 \alpha -1} |v| |\nabla v| dx & \leq & \| |x|^\alpha |v| \|_{L^2} \| |x|^{\alpha-1} | \nabla v | \|_{L^2} \\
& \leq & \| |x|^\alpha |v| \|_{L^2}^{2-\frac{1}{\alpha}} \| v \|_{H^\alpha}^{\frac{1}{\alpha}}.
\end{eqnarray*}
Hence, using $(ii)$
\begin{eqnarray*}
\frac{d}{dt} [ \| |x|^\alpha |v(t)| \|^2_2 ] \lesssim \| |x|^\alpha |v| \|_{L^2}^{2-\frac{1}{\alpha}} \| \phi \|_{H^\alpha}^{\frac{1}{\alpha}} + \| v \|_{L^2}^2.
\end{eqnarray*}
Integrating, we have
\begin{eqnarray*}
\| |x|^\alpha |v| \|^2_{L^2 L^\infty ([0,t])} & \lesssim & \| |x|^\alpha |\phi| \|^2_2 + \int_0^t [ \| |x|^\alpha v (s) \|_{L^2}^{2-\frac{1}{\alpha}} \| \phi \|_{H^\alpha}^{\frac{1}{\alpha}} + \| v(s) \|_{L^2}^2 ]ds \\
& \lesssim & \| |x|^\alpha |\phi| \|^2_2 + \| |x|^\alpha |v| \|^{2-\frac{1}{\alpha}}_{L^2 L^\infty ([0,t])} \int_0^t [ \| \phi \|_{H^\alpha}^{\frac{1}{\alpha}} + \| \phi \|_{L^2}^2] ds \\
& \lesssim & \| |x|^\alpha |\phi| \|^2_2 + \epsilon \| |x|^\alpha |v| \|^2_{L^2 L^\infty ([0,t])} + C(\epsilon) (t^{2 \alpha} + t) \| |x|^\alpha |\phi| \|^2_2 .
\end{eqnarray*}
Hence, estimate $(iv)$ follows.
\end{proof}

\section{Strichartz Estimates}
\label{lin:strich}
From the above time decay, we can also prove the standard space-time Strichartz estimates for $e^{i \mathcal{H} t} \phi$ where $\phi \in P_c \mathcal{H}$.  We review the standard methods here as seen in \cite{SS}.  From henceforward, let us assume that we work on the subspace of functions contained in $P_c \mathcal{H}$.

\begin{thm}
\label{thm:strich1}
For $p$ and $p'$ such that $\frac{1}{p} + \frac{1}{p'} = 1$, with $2 \leq p \leq \infty$, and $t \neq 0$, the transformation $e^{i \mathcal{H} t}$ maps continuously $L^{p'} ( \R^d)$ into $L^{p} ( \R^d)$ and
\begin{eqnarray}
\label{eqn:strich1}
\| e^{i \mathcal{H} t} \phi \|_{L^p} \lesssim \frac{1}{|t|^{d(\frac{1}{2}-\frac{1}{p})}} \| \phi \|_{L^{p'}}.
\end{eqnarray}
\end{thm}

\begin{proof}
This result follows from the interpolation result presented in \cite{BL}.
\end{proof}

\begin{defn}
The pair $(q,r)$ of real numbers is called admissible if $\frac{2}{q} = \frac{d}{2} - \frac{d}{r}$ with $2 \leq r < \frac{2d}{d-2}$ when $d > 2$, or $2 \leq r \leq \infty$ when $d = 1$ or $d = 2$.
\end{defn}

The following result proving Strichartz estimates is from \cite{S1}.

\begin{thm}[Schlag]
\label{thm:strich2}
For every $\phi \in L^2$ and every admissible pair $(q,r)$, the function
$t \to e^{i \mathcal{H} t} \phi$ belongs to $L^q (\reals, L^r (\reals^d)) \cap C ( \reals, L^2 (\reals^d))$, and there exists a constant $C$ depending only on $q$ such that
\begin{eqnarray}
\label{eqn:strich2}
\| e^{i \mathcal{H} t} \phi \|_{L^q (\reals, L^r (\reals^d))} \leq C \| \phi \|_{L^2}.
\end{eqnarray}
\end{thm}

\begin{proof}
Typically, one uses a duality argument when the operator $e^{i \mathcal{H} t}$ is unitary.  Namely, 
\begin{eqnarray*}
| \langle e^{i \mathcal{H} t} \phi, G \rangle_{L^2 (\reals^{d+1})} | \lesssim \| \phi \|_{L^2} \| G \|_{L^{q'} L^{r'}}.
\end{eqnarray*}
To this end, write
\begin{eqnarray*}
| \int_{-\infty}^\infty \langle e^{i \mathcal{H} t} \phi, G \rangle_{L^2 (\reals^{d})}  ds | & = & | \left\langle \phi, \int_{-\infty}^\infty e^{i \mathcal{H} -s} G(s) ds \right\rangle_{L^2 (\reals^d)} | \\
& \leq & \| \phi \|_{L^2 (\reals^d)} \left\| \int_{-\infty}^\infty e^{i \mathcal{H} -s} G(s) ds \right\|_{L^2 (\reals^d)},
\end{eqnarray*}
where
\begin{eqnarray*}
\left\| \int_{-\infty}^\infty e^{i \mathcal{H} -s} G(s) ds \right\|_{L^2 (\reals^d)}^2 & = & \left\langle \int_{-\infty}^\infty e^{i \mathcal{H} -s} G(s) ds , \int_{-\infty}^\infty e^{i \mathcal{H} -t} G(t) dt \right\rangle_{L^2 (\reals^d)} \\
& = & \left\langle  \int_{-\infty}^\infty G(t) dt, \int_{-\infty}^\infty e^{i \mathcal{H} t -s} G(s) ds \right\rangle_{L^2 (\reals^d)} \\
& \leq & \| G \|_{L^{q'} L^{r'}} \left\| \int_{-\infty}^\infty e^{i \mathcal{H} \cdot -s} G(s) ds \right\|_{L^q L^r}.
\end{eqnarray*}

Using Equation \ref{eqn:strich1}, we have
\begin{eqnarray*}
 \left\| \int_{-\infty}^\infty e^{i \mathcal{H} t -s} G(s) ds \right\|_{L^r}  & \leq & \int_{-\infty}^\infty \left\|  e^{i \mathcal{H} t -s} G(s)  \right\|_{L^r} ds \\
& \leq &  \int_{-\infty}^\infty \frac{1}{|t-s|^{d(\frac{1}{2}-\frac{1}{r})}} \left\| G(s)  \right\|_{L^{r'}} ds \\
& \leq & \int_{-\infty}^\infty \frac{1}{|t-s|^{\frac{2}{q}}}   \left\| G(s)  \right\|_{L^{r'}} ds.
\end{eqnarray*}
Hence, using the Hardy-Littlewood-Sobolev Theorem with $\gamma = -\frac{2}{q}$, 
\begin{eqnarray*}
 \left\| \int_{-\infty}^\infty e^{i \mathcal{H} t -s} G(s) ds \right\|_{L^q L^r} \lesssim \| G \|_{L^{q'} L^{r'}}.
\end{eqnarray*}

However, for systems, this is not applicable.  Hence, we must use the Christ-Kiselev Lemma \cite{ChKi}.

\begin{lem}
Let $X$, $Y$ be Banach spaces and let $K(t,s)$ be the kernel of the operator
\begin{eqnarray*}
K: L^p ([0,T];X) \to L^q ([0,T];Y).
\end{eqnarray*}
Denote by $\| K \|$ the operator norm of $K$.  Define the lower diagonal operator
\begin{eqnarray*}
\tilde{K} : L^p ([0,T];X) \to L^q ([0,T];Y)
\end{eqnarray*}
to be
\begin{eqnarray*}
\tilde{K} f(t) = \int_0^t K(t,s) f(s) ds.
\end{eqnarray*}
Then, the operator $\tilde{K}$ is bounded from $L^p ([0,T];X) \to L^q ([0,T];Y)$ and it norm $\| \tilde{K} \| \leq c \| K \|$, provided $p < q$.
\end{lem}

A perturbative approach originated by Kato is used.  Define
\begin{eqnarray*}
(S F) (t,x) = \int_0^t (e^{-i(t-s) \mathcal{H}} P_c F(s,\cdot)) (x) ds.
\end{eqnarray*}
Then,
\begin{eqnarray*}
\| S F \|_{L^\infty_t L^2_x} \lesssim \| F \|_{L^1_t L^2_x}.
\end{eqnarray*}
Using the fractional integration argument from the unitary case, we have
\begin{eqnarray*}
\| SF \|_{L^r_t L^p_s} \lesssim \| F \|_{L^{r'}_t L^{q'}_x},
\end{eqnarray*}
where $(r,p)$ is admissible.  By Duhamel, we have
\begin{eqnarray*}
e^{-it\mathcal{H}} P_c = e^{-it \mathcal{H}_0} P_c - i \int_0^t e^{-i (t-s) \mathcal{H}_0} V e^{-is \mathcal{H}} P_c ds.
\end{eqnarray*}
Set $V = \tilde{M} \tilde{M}^{-1} V$, where 
\begin{eqnarray*}
\tilde{M} = \left[ \begin{array}{cc}
\langle x \rangle^{-1-} & 0 \\
0 & \langle x \rangle^{-1-}
\end{array} \right].
\end{eqnarray*}
Then,
\begin{eqnarray*}
\left\| \int_0^\infty e^{-i (t-s) \mathcal{H}_0} \tilde{M} g(s) ds \right\|_{L^r_t L^p_x} \lesssim \left\| \int_0^\infty e^{is \mathcal{H}_0} \tilde{M} g(s) \right\|_{L^2} \lesssim \| g \|_{L^2_t L^2_x},
\end{eqnarray*}
where the last inequality follows from local smoothing.  Applying the Christ-Kiselev lemma, for any Strichartz pair $(r,p)$, we have
\begin{eqnarray*}
\left\| \int_0^t e^{-i (t-s) \mathcal{H}_0} \tilde{M} g(s) ds \right\|_{L^r_t L^p_x} \lesssim \| g \|_{L^2_t L^2_x}.
\end{eqnarray*}
Then,
\begin{eqnarray*}
\left\| e^{-it\mathcal{H}} P_c f \right\|_{L^r_t L^p_x} \lesssim \| f \|_{L^2} + \left\| \tilde{M}^{-1} V e^{-is \mathcal{H}} P_c f \right\|_{L^2_s L^2_x},
\end{eqnarray*}
so we need
\begin{eqnarray*}
\left\| \tilde{M}^{-1} V e^{-is\mathcal{H}} P_c f \right\|_{L^2_s L^2_x} \lesssim \| f \|_{L^2}.
\end{eqnarray*}
Taking a Fourier transform in $s$ gives
\begin{eqnarray*}
\int_{-\infty}^\infty \| \tilde{M}^{-1} V [P_c (\mathcal{H} - \lambda - i0) P_c]^{-1} P_c f \|_{L^2}^2 d\lambda \lesssim \| f \|_{L^2}^2.
\end{eqnarray*}
However, this follows from the smoothing estimate on $\mathcal{H}_0$, plus the standard resolvent identity under the spectral assumptions on $\mathcal{H}$.  Hence,
\begin{eqnarray*}
\| e^{-it\mathcal{H}} P_c f \|_{L^r_t L^p_x} \lesssim \| f \|_{L^2}.
\end{eqnarray*}
\end{proof}

\end{document}